\pdfoutput=1
\documentclass[12pt]{paper}

\usepackage{color}
\usepackage{mathptmx}       % selects Times Roman as basic font
\usepackage{helvet}         % selects\textbf{\textbf{•}} Helvetica as sans-serif font
\usepackage{courier}        % selects Courier as typewriter font
\usepackage{type1cm}        % activate if the above 3 fonts are
\usepackage{makeidx}         % allows index generation
\usepackage{graphicx}        % standard LaTeX graphics tool
\usepackage{multicol}        % used for the two-column index
\usepackage[bottom]{footmisc}% places footnotes at page bottom
\usepackage{bbm}
\usepackage{amsmath}
\usepackage{amsfonts}
\usepackage{amssymb, eucal}
\usepackage{amsthm}
\usepackage{amsopn}
\usepackage{subcaption}
\usepackage{comment}

\usepackage{times,upref, mathrsfs, xcolor, dsfont}
\usepackage{amsfonts,amsmath,amstext,amsbsy, amsopn}

\newtheorem{theorem}{Theorem}
\newtheorem{corollary}[theorem]{Corollary}

\newtheorem{lemma}[theorem]{Lemma}

\newtheorem{proposition}[theorem]{Proposition}

\newtheorem{remark}[theorem]{Remark}
\newtheorem{definition}[theorem]{Definition}

\newtheorem{notation}[theorem]{Notation}
\usepackage[top=2cm, bottom=2cm, left=2cm, right=2cm]{geometry}
\numberwithin{theorem}{section}
\numberwithin{figure}{section}
\numberwithin{equation}{section}

\begin{document}

\title{Conditional measures of generalized Ginibre point processes}

\author{Alexander I. Bufetov and Yanqi Qiu}
%\thanks{ }

% your contribution title if the original one is too long%

\maketitle

\abstract{The main result of this paper is that conditional measures of generalized Ginibre point processes, with respect to the configuration in the complement of a bounded open subset  on
$\mathbb{C}$, are orthogonal polynomial ensembles with weights found explicitly.
An especially simple formula for conditional measures is obtained in the particular
case of radially-symmetric determinantal point processes, including the classical Ginibre point process.\\

}

\newcommand{\R}{\mathbb{R}}
\newcommand{\N}{\mathbb{N}}
\newcommand{\Z}{\mathbb{Z}}
\newcommand{\Q}{\mathbb{Q}}
\newcommand{\C}{\mathbb{C}}
\newcommand{\E}{\mathbb{E}}
\newcommand{\PP}{\mathbb{P}}
\newcommand{\F}{\mathbb{F}}

\newcommand{\X}{\mathscr{X}}
\newcommand{\M}{\mathcal{M}}
\newcommand{\ZZ}{\mathcal{Z}}
\newcommand{\K}{\mathcal{K}}

\newcommand{\Conf}{\mathrm{Conf}}
\newcommand{\tr}{\mathrm{tr}}
\newcommand{\rank}{\mathrm{rank}}
\newcommand{\Var}{\mathrm{Var}}
\newcommand{\supp}{\mathrm{supp}}
\newcommand{\esssup}{\mathrm{esssup}}
\newcommand{\sgn}{\mathrm{sgn}}

\newcommand{\Det}{\mathrm{det}}
\newcommand{\Ran}{\mathrm{Ran}}

\newcommand{\an}{\text{\, and \,}}
\newcommand{\as}{\text{\, as \,}}

\newcommand{\ch}{\mathbbm{1}}

\section{Introduction}
\subsection{Outline of the main results}
Let $\phi: \C\rightarrow\R$ be a real function. Under some additional assumptions, one can assign to $\phi$ the generalized Fock space $\mathscr{F}_\phi$ of holomorphic functions on $\C$, square integrable with respect to the measure $d\lambda_{\phi}(z)=e^{-2\phi(z)}d\lambda(z)$, where $d\lambda$ is the Lebesgue measure on $\C$. The orthogonal projection operator $\Pi\colon L^2\left(\mathbb{C},d\lambda_{\phi}\right)\rightarrow \mathscr{F}_\phi$ induces a determinantal measure $\mathbb{P}_\Pi$ on the space of configurations on $\mathbb{C}$. For example, for $\phi(z)={|z|}^2$ one obtains the classical Ginibre point process of random matrix theory. In this paper we describe, for the point process $\mathbb{P}_\Pi$, conditional measures in a bounded domain $B$ with respect to the fixed configuration in the exterior $\mathbb{C}\setminus B$.

In Theorem \ref{cond-general} below, under some additional assumptions we show that these conditional measures are orthogonal polynomial ensembles of the form
\begin{equation}
\label{cond-prelim}
Z^{-1}\prod_{1\leqslant i<j\leqslant N}\left|z_i-z_j\right|^2\cdot\prod_{i=1}^N \rho\left(z_i\right)d\lambda(z_i),
\end{equation}
where $Z$ is the normalization constant and the weight $\rho$ is found explicitly as a function of the fixed configuration $\X\setminus B  = \X \cap (\C \setminus B)$. Theorem \ref{cond-general} is an analogue of the Gibbs property for our processes (see e.g. Sinai \cite{sinai}).

 In particular, if the function $\phi$ is radial, i.e., only depends on $|z|$, and the domain $B$ contains $0$, then we have
\[
\rho(z)=\prod_{x\in \X \backslash B}\left|1-\frac{z}{x}\right|^2\cdot\frac{d\lambda_{\phi}}{d\lambda}(z),
\]
where the product is taken over the fixed particles of our configuration $\X$ in $\mathbb{C}\backslash B$ and understood in principal value, see Corollary  \ref{cond-radial} below.

The proof of Theorem \ref{cond-general}  follows the general scheme, developed in  \cite{Buf-gibbs}, \cite{buf-cond} for point processes on $\mathbb{R}$,  of the computation of conditional measures in intervals with respect to fixed exterior and relies on the results of \cite{QB3} on Palm measures and quasi-symmetries of determinantal point processes corresponding to Hilbert spaces of holomorphic functions (see \cite{Buf-gibbs}, \cite{GO-Adv} for more background on quasi-symmetries of determinantal point processes). Regularization of multiplicative functionals requires extra effort in the complex case since we must work with the von~Neumann-Schatten class $\mathscr{C}_3$ instead of the space of Hilbert-Schmidt operators.

The main part  of the proof of Corollary \ref{cond-radial} is the explicit computation of the normalization constants in the radial case. This explicit computation is given in Theorem \ref{thm-main1}.
The proof of Theorem \ref{thm-main1} proceeds by finite-dimensional approximation, the main difficulty
coming from the need for  careful estimates of conditionally convergent series, see especially Lemma \ref{lem-E4}   .

\subsection{Formulation of the main results}

Recall that a configuration $ \X$ on the complex plane $\C$ is a locally finite subset $\X \subset \C$ or, equivalently, the corresponding counting measure.  The space of configurations $\Conf(\C)$ is a subset of the space $\mathfrak{M}(\C)$ of Radon measures on $\C$ and itself a complete separable metric space (cf. e.g. Kallenberg \cite{kal1}, \cite{kal2}). We equip  $\Conf(\C)$ with its Borel sigma algebra.  Given a bounded Borel subset $B \subset \C$ and a configuration $\X \in \Conf(\C),$ let $\#_B(\X)$ stand for the number of particles of $\X$ lying in $B$; the random variables $\#_B$ over all bounded Borel subsets $B$ generate the Borel sigma-algebra.   Given a Borel subset $W \subset \C,$ we let $\mathcal{F}_W$ be the $\sigma$-algebra generated by all random variables of the form $\#_B$ with $B$ ranging over all Borel subsets of $W$.
Given a configuration $\X$ and a subset $W$ of $\C$, we write $\X|_W$ for the restriction of $\X$ onto the subset $W$.

In this paper we consider  point processes on the complex plane $\C$, i.e.  Borel probability measures on the space $\Conf(\C)$ of configurations of $\C$.
For such a measure $\PP$ and a Borel subset $W \subset \C$, the measure $\PP(\cdot | \X; W)$ on $\Conf(\C \setminus W)$ is defined for $\PP$-almost every configuration $\X$ as the conditional measure of $\PP$ with respect to the condition that the restriction of our random configuration onto $W$ coincides with $\X|_W$. More formally, consider the surjective restriction mapping $\X \to \X|_W$ from
$\Conf(\C)$ to $\Conf(W)$. Fibres of this mapping can be identified with $\Conf(\C\backslash W)$, and conditional measures, in the sense of Rohlin \cite{Rohmeas},  are precisely the measures $\PP(\cdot | \X; W)$.
If the point process $\PP$ admits correlation measures of order up to $\ell$, then, given distinct points $q_1, \dots, q_\ell \in \C$, we let $\PP^{q_1, \dots, q_\ell}$ stand for the $\ell$-th reduced Palm measure of $\PP$ conditioned at points $q_1, \dots, q_\ell$ (here and below we follow the conventions of \cite{Buf-gibbs} in working with Palm measures).

Recall that for any $q\in\C$, we define
\[
\Pi^{q}(x,y) = \Pi(x,y) - \frac{\Pi(x,q) \Pi(q,y)}{ \Pi(q,q)}.
\]
More generally,  for an $\ell$-tuple $\mathfrak{q} = (q_1, \cdots, q_\ell)$ of distinct points in $\C$, we define $\Pi^{\mathfrak{q}} = (\cdots (\Pi^{q_1})^{q_2} \cdots)^{q_\ell}$. The Shirai-Takahashi Theorem \cite{ShirTaka1} asserts that   $\PP_\Pi^{\mathfrak{q}} = \PP_{\Pi^\mathfrak{q}}$.

Throughout this paper, we fix a $C^2$-function $\phi: \C \rightarrow \R$. We equip  the complex plane $\C$ with the measure $d\lambda_\phi(z) = e^{-2\phi(z)} d\lambda(z)$, where $d\lambda$ is the Lebesgue measure. We always assume that there exist positive constants $m, M > 0$ such that
\begin{align}\label{sub-h}
m \le \Delta \phi \le M,
\end{align}
where $\Delta$ is the Euclidean Laplacian. Denote by $\mathscr{F}_\phi$ the generalized Fock space with respect to the weight $e^{-2 \phi(z)}$ and let  $\Pi$ be the reproducing kernel of $\mathscr{F}_\phi$.
Let $\PP_{\Pi}$ be the  determinantal measure on $\Conf(\C)$ corresponding to the kernel $\Pi$ considered with respect to the reference measure $d\lambda_{\phi}(z)$ on the phase space $\C$ ( see  e.g. \cite{Buf-gibbs}, \cite{QB3}, \cite{Macchi}, \cite{ShirTaka1}, \cite{ShirTaka2}, \cite{Soshnikov} for the background on spaces of configurations and  determinantal point processes).

For any $\ell \in\N$ and any  two $\ell$-tuples $\mathfrak{p} = (p_1, \dots, p_\ell)$ and $\mathfrak{q} = (q_1, \dots, q_\ell)$ of distinct points in $\C$, we fix  a positive number $r_{\mathfrak{p}, \mathfrak{q}}>0$, continuously depending on $\mathfrak{p}, \mathfrak{q}$ and large enough in such a way that
\begin{align}\label{close-1}
\sup_{|z|\ge r_{\mathfrak{p}, \mathfrak{q}}} \Big|  \Big| \frac{z-p_i}{z- q_i }\Big|^2 -1  \Big|\le  1/2 \text{ for all $1\le i \le \ell$}  \an \sup_{|z|\ge r_{\mathfrak{p}, \mathfrak{q}}} \Big|  \prod_{i= 1}^\ell\Big| \frac{z-p_i}{z- q_i }\Big|^2 -1  \Big|\le  1/2.
\end{align}
For any $p\in \C$ and $z\in \C^*$ set
\begin{align}\label{def-kappa}
\kappa(p, z) : = \frac{p}{z} +  \frac{\bar{p}}{\bar{z}} +   \frac{p^2}{2 z^2}  +  \frac{\bar{p}^2}{2\bar{z}^2}
\end{align}
and for an $\ell$-tuple  $\mathfrak{p} = (p_1, \dots, p_\ell)$ of distinct points in $\C$ write
\begin{align}\label{def-kappa-g}
\kappa(\mathfrak{p}, z) : =  \sum_{i =1}^\ell \kappa(p_i, z).
\end{align}

Recall that the tail sigma-algebra  consists of those Borel subsets of $\Conf(\C)$ that, for any bounded Borel $B$, belong to the sigma-algebra $\mathcal{F}_{\C\setminus B}$.

\begin{proposition}\label{prop-im} Let $\mathfrak{p} = (p_1, \dots, p_\ell)$ and $\mathfrak{q} = (q_1, \dots, q_\ell)$ be  two $\ell$-tuples of distinct points in $\C$.
\begin{itemize}
\item[(i)]  The limit
\begin{align}\label{general-tuple}
\Psi_{\mathfrak{p}, \mathfrak{q}}(\X)  =  \lim_{R\to\infty} \exp\Big(   \int\limits_{r_{\mathfrak{p}, \mathfrak{q}} \le |z|\le R}    (  \kappa(\mathfrak{p}, z) - \kappa(\mathfrak{q}, z) )  \Pi(z,z)  d\lambda_\phi(z) \Big)   \prod_{x\in\X: | x| \le R}   \prod_{i= 1}^\ell\Big| \frac{x-p_i}{x- q_i }\Big|^2
\end{align}
exists in $L^1(\Conf(\C), \PP_{\Pi}^{\mathfrak{q}})$.
\item[(ii)] The Palm measures $\PP_{\Pi}^{\mathfrak{p}}$ and $\PP_{\Pi}^{\mathfrak{q}}$ are in the same measure class.  The Radon-Nikodym derivative $d \PP_{\Pi}^{\mathfrak{p}}/ d \PP_{\Pi}^{\mathfrak{q}}$ is given by
\begin{equation}\label{rdder}
\frac{d \PP_{\Pi}^{\mathfrak{p}}}{ d \PP_{\Pi}^{\mathfrak{q}} } (\X) = \frac{\Psi_{\mathfrak{p}, \mathfrak{q}} (\X) }{  \displaystyle{\int\limits_{\Conf(\C)} \Psi_{\mathfrak{p}, \mathfrak{q}} d \PP_{\Pi}^{\mathfrak{q}}}}.
\end{equation}
Consequently, $\Psi_{\mathfrak{p}, \mathfrak{q}} (\X)$ is positive for $\PP_\Pi^{\mathfrak{q}}$-almost every configuration $\X$.  
\item[(iii)] There exists a Borel subset $\mathcal{W} \subset \Conf(\C)$  belonging to the tail $\sigma$-algebra and satisfying $\PP_{\Pi}(\mathcal{W} ) = 1$  such that for any bounded subset $K\subset \C$,  there exists a subsequence $R_n\to\infty$, along which the convergence in \eqref{general-tuple} takes place uniformly for all  $\ell$-tuples $\mathfrak{p}, \mathfrak{q}$ of distinct points in $K$ and $\X \in \mathcal{W}$.  The mapping $$(\mathfrak{p}, \mathfrak{q})\rightarrow  \Psi_{\mathfrak{p}, \mathfrak{q}}(\X)$$ is continuous on $\C^\ell \times (\C\setminus \X)^{\ell}$ for  every
configuration $\X\in {\mathcal W}$.

\item[(iv)] The function $\displaystyle{(\mathfrak{p}, \mathfrak{q})\rightarrow \int_{\Conf(\C)}  \Psi_{\mathfrak{p}, \mathfrak{q}}d \PP_{\Pi}^{\mathfrak{q}}}$ is continuous on $\C^\ell \times \C^{\ell}$.

\end{itemize}
\end{proposition}

\begin{corollary}\label{cor-radial}
Let $\mathcal W$ be as in item (iii) of Proposition \ref{prop-im}.
Assume that $\phi$ is radial and satisfies  \eqref{sub-h}. Then for any two $\ell$-tuples $\mathfrak{p} = (p_1, \dots, p_\ell)$ and $\mathfrak{q} = (q_1, \dots, q_\ell)$ of distinct points in $\C$, the limit
\begin{align}\label{general-tuple-radial}
\Gamma_{\mathfrak{p}, \mathfrak{q}}(\X)  =  \lim_{R\to\infty}  \prod_{x\in\X: | x| \le R}   \prod_{i= 1}^\ell\Big| \frac{x-p_i}{x- q_i }\Big|^2
\end{align}
exists for any $\X \in \mathcal{W}$ and in $L^1(\Conf(\C), \PP_{\Pi}^{\mathfrak{q}})$ .  Moreover, the function $(\mathfrak{p}, \mathfrak{q})\rightarrow  \Gamma_{\mathfrak{p}, \mathfrak{q}}(\X)$ is continuous on $\C^\ell \times (\C\setminus \X)^{\ell}$ for every configuration $\X\in \mathcal W$.
\end{corollary}

For example, for $p\in\C$, the limit
\begin{equation}\label{gammapzero}
\Gamma_{p, 0} (\X)  = \lim_{R\to\infty}  \prod_{x\in\X: | x| \le R}   \Big| 1- \frac{p}{x}\Big|^2
\end{equation}
exists  for any $\X\in \mathcal{W}$ and in $L^1(\Conf(\C), \PP_{\Pi}^{0})$.

\begin{proof}[Proof of Corollary \ref{cor-radial}]
 If $\phi: \C \rightarrow \R$ is radial, then for any $r$  satisfying $0 < r < R$ we have
\[
 \int\limits_{r \le |z|\le R}    \kappa(\mathfrak{p} , z)  \Pi(z,z)  d\lambda_\phi(z)  =  \int\limits_{r \le |z|\le R}    \kappa(\mathfrak{q}, z)  \Pi(z,z)  d\lambda_\phi(z)  = 0.
\]
Corollary \ref{cor-radial} follows now from Proposition \ref{prop-im}.
\end{proof}

\begin{theorem}\label{thm-main1}
Assume that $\phi$ is radial and satisfies  \eqref{sub-h}. Then for any two $\ell$-tuples $\mathfrak{p} = (p_1, \dots, p_\ell)$ and $\mathfrak{q} = (q_1, \dots, q_\ell)$ of distinct points in $\C$, we have
\begin{align}\label{com-exp}
\E_{\PP_{\Pi}^{\mathfrak{q}}} [\Gamma_{\mathfrak{p}, \mathfrak{q}}] = \int\limits_{\Conf(\C)} \Gamma_{\mathfrak{p}, \mathfrak{q}}(\X) d \PP_{\Pi}^{\mathfrak{q}} (\X) =    \frac{ \Det_{i, j =1}^\ell (\Pi(p_i, p_j))}{ \Det_{i, j =1}^\ell(\Pi(q_i, q_j))}
\displaystyle \prod\limits_{1 \leqslant i< j \leqslant \ell} \left|\frac{q_i- q_j}{ p_i- p_j}\right|^2.
\end{align}
\end{theorem}

{\bf {Remark}.}
Osada and Shirai \cite{Osada-Shirai} obtained the results in Corollary \ref{cor-radial} and Theorem  \ref{thm-main1} for the special case $\phi(z) = | z|^2$ (corresponding to the standard Ginibre point process).

\begin{definition}
Define  a positive function $\rho_\Pi: \C \rightarrow \R$ by
\begin{align}\label{exp-formula}
\int\limits_{\Conf(\C)} \Psi_{p, q}(\X) d \PP_{\Pi}^q (\X)  =     \frac{\rho_\Pi(q) }{\rho_\Pi(p)} \frac{\Pi(p, p)}{\Pi(q, q)}.
\end{align}
\end{definition}

In particular, for any $p, q \in \C$,  we have
\[
\frac{d \PP_{\Pi}^p}{ d \PP_{\Pi}^q} (\X) =   \frac{\rho_\Pi(p)}{\rho_\Pi(q)}   \frac{\Pi(q,q)}{\Pi(p,p)} \Psi_{p, q}(\X).
\]

\begin{theorem}\label{cond-general}
Let $B\subset \C$ be a bounded set.
For $\PP_{\Pi}$-almost every $\X\in
\Conf(\C)$, the  measure $\PP_{\Pi}(\cdot | \X; \C \setminus B)$ has the form
\begin{equation}\label{cond-la-general}
Z(B,\X)^{-1} \prod\limits_{1\le i<j  \le \#_B(\X)} |z_i-z_j|^2 \prod \limits_{i=1}^{\#_B(\X)}\rho_{B,\X}(z_i)d\lambda_{\phi}(z_i),
\end{equation}
where  $\#_B(\X)$ stand for the number of particles of $\X$ lying in $B$ (which is measurable with respect to $\X|_{{\mathbb C}\setminus B}$) ;    $Z(B, \X)$ is the normalization constant
and the function $\rho_{B, \X}$ satisfies, for any $p,q\in B$, the relation
\begin{equation}\label{rhopix-la}
\frac{\rho_{B, \X}(p)}{\rho_{B, \X}(q)}=\frac{\rho_\Pi(p)}{\rho_\Pi(q)}
{\Psi}_{p,q}(\X|_{{\mathbb C}\setminus B}).
\end{equation}
\end{theorem}

An especially simple expression is obtained for radially-symmetric weights.

\begin{corollary}\label{cond-radial}
Assume that $\phi$ is radial and that $B$ contains the origin $0$. Then for $\PP_{\Pi}$-almost every $\X\in
\Conf(\C)$, the  measure $\PP_{\Pi}(\cdot | \X; \C \setminus B)$ has the form
\begin{equation}\label{expform-cond-la}
Z(B, \X)^{-1} \prod\limits_{1\leq i<j\leq \#_B(\X)} |z_i-z_j|^2 \prod \limits_{i=1}^{\#_B(\X)}   {\Gamma}_{z_i, 0}(\X|_{{\mathbb C}\setminus B})d\lambda_{\phi}(z_i),
\end{equation}
where the functions ${\Gamma}_{z_i, 0}$ are defined by (\ref{gammapzero}).
\end{corollary}

\begin{remark}
Formulae \eqref{cond-la-general} and \eqref{expform-cond-la} can be viewed as the analogue of the Dobrushin-Lanford-Ruelle equation in our situation. We are deeply grateful to the anonymous referee for suggesting to add this remark.
\end{remark}

\subsection{Derivation  of Theorems \ref{cond-general} and Corollary \ref{cond-radial} from Proposition \ref{prop-im}.}
 
We now recall, for our particular case, the Ghosh and Peres \cite{Ghosh-sine}, \cite{Ghosh-rigid} definition of rigidity (see  Holroyd-Soo \cite{hsoo}  and \cite{Buf-rigid}, \cite{BDQ} for further background and results on rigidity of point processes). Given a Borel subset $W$ of $\C$, write $\mathcal{F}_W^{\PP}$ for the $\PP$-completion of $\mathcal{F}_W$. A point process $\PP$ on $\C$ is  {\it rigid} if for any bounded Borel  subset $B \subset \C$ the function $\#_B$ is  $\mathcal{F}_{\C \backslash B}^{\PP}$-measurable.  As established in \cite{QB3}, Proposition 1.2,  our point process $\PP_\Pi$ is rigid in the sense of Ghosh and Peres.
For a subset $B\subset \C$ and a natural number $\ell$, we write $\Conf_\ell(B)$ for the space of $\ell$-particle configurations on $B$; in other words, the space of all subsets of $B$ of cardinality $\ell$.
Rigidity implies that for  any precompact Borel set $B\subset \C$ and $\PP$-almost any $\X$ the conditional measure
$\PP(\cdot |\X; \C\setminus B)$ is supported on the subset $\Conf_\ell(B)$, where $\ell=\#_B(\X)$.

Next, we use the characterization of conditional measures in terms of Radon-Nikodym derivatives of Palm measures of the same order established in Proposition 3.1 in \cite{buf-cond}.
Together with rigidity,  Proposition 3.1 in \cite{buf-cond} implies that, for our point processes,
the conditional measure $\mathbb{P}(\cdot | \X; W)$ has the form
\begin{equation}\label{cond-expl}
Z^{-1}(q_1, \dots, q_\ell)\frac{d\PP^{p_1, \dots, p_\ell}}{d\PP^{q_1, \dots, q_\ell}}
\left( \X|_W\right) d{\rho}_\ell(p_1, \dots, p_\ell),
\end{equation}
where $q_1, \dots, q_\ell$ is almost any fixed $\ell$-tuple, $\rho_\ell$ is the $\ell$-th correlation measure of $\PP$  and $Z(q_1, \dots, q_\ell)$ is the normalization constant.

Item (ii) of Proposition \ref{prop-im} gives precisely the explicit expression for Radon-Nikodym derivatives of Palm measures of the same order, and, consequently, Proposition \ref{prop-im} immediately implies Theorem \ref{cond-general}, Corollary \ref{cond-radial}.
We proceed to the proof of Proposition \ref{prop-im}.

%
%Recall that the (reduced) Palm measure at a point $p \in \C$ is induced by a kernel $\Pi^p$, which we call the Palm kernel at $p$:
%\begin{align}\label{palm-k}
%\Pi^p(z,w)  = \Pi (z,w)  - \frac{\Pi(z, p) \Pi(p, w)}{\Pi(p,p)}.
%\end{align}

%%%%%%%%%%%%%%%%%%%%%%%%%%%%%%%%%%%%%%%
%%%%%%%%%%%%%%%%%%%%%%%%%%%%%%%%%%%%%%%
%%%%%%%%%%%%%%%%%%%%%%%%%%%%%%%%%%%%%%%
%%%%%%%%%%%%%%%%%%%%%%%%%%%%%%%%%%%%%%%
%%%%%%%%%%%%%%%%%%%%%%%%%%%%%%%%%%%%%%%
%%%%%%%%%%%%%%%%%%%%%%%%%%%%%%%%%%%%%%%
%%%%%%%%%%%%%%%%%%%%%%%%%%%%%%%%%%%%%%%
%%%%%%%%%%%%%%%%%%%%%%%%%%%%%%%%%%%%%%%
%%%%%%%%%%%%%%%%%%%%%%%%%%%%%%%%%%%%%%%
%%%%%%%%%%%%%%%%%%%%%%%%%%%%%%%%%%%%%%%
%%%%%%%%%%%%%%%%%%%%%%%%%%%%%%%%%%%%%%%
%%%%%%%%%%%%%%%%%%%%%%%%%%%%%%%%%%%%%%%
%%%%%%%%%%%%%%%%%%%%%%%%%%%%%%%%%%%%%%%
%%%%%%%%%%%%%%%%%%%%%%%%%%%%%%%%%%%%%%%
%%%%%%%%%%%%%%%%%%%%%%%%%%%%%%%%%%%%%%%
%%%%%%%%%%%%%%%%%%%%%%%%%%%%%%%%%%%%%%%

\section{Regularized multiplicative functionals}

We first collect some results from \cite{QB3}. 
\begin{proposition}\label{prop-ref}
Let $\mathfrak{p} = (p_1, \dots, p_\ell)$ and $\mathfrak{q} = (q_1, \dots, q_\ell)$ be two $\ell$-tuples of distinct points in $\C$. Then
\begin{itemize}
\item[(i)] The limit
\begin{align}\label{ref-tuple}
\widetilde{\Psi}_{\mathfrak{p}, \mathfrak{q}}(\X)  =  \lim_{R\to\infty} \exp\Big(  -  2  \sum_{i =1}^\ell \int\limits_{|z|\le R}    \log\Big |  \frac{z-p_i}{ z- q_i} \Big| \Pi^{\mathfrak{q}}(z,z)  d\lambda_\phi(z) \Big)  \prod_{x\in\X: | x| \le R}   \prod_{i= 1}^\ell\Big| \frac{x-p_i}{x- q_i }\Big|^2
\end{align}
exists in $L^1(\Conf(\C), \PP_{\Pi}^{\mathfrak{q}})$.
\item[(ii)]  The Palm measures $\PP_{\Pi}^{\mathfrak{p}}$ and $\PP_{\Pi}^{\mathfrak{q}}$ are in the same measure class. The Radon-Nikodym derivative $d \PP_{\Pi}^{\mathfrak{p}}/ d \PP_{\Pi}^{\mathfrak{q}}$ is given by
\begin{align*}
\frac{d \PP_{\Pi}^{\mathfrak{p}}}{ d \PP_{\Pi}^{\mathfrak{q}} } (\X) = \frac{\widetilde{\Psi}_{\mathfrak{p}, \mathfrak{q}} (\X) }{ \displaystyle{  \int\limits_{\Conf(\C)} \widetilde{\Psi}_{\mathfrak{p}, \mathfrak{q}} d \PP_{\Pi}^{\mathfrak{q}}}}.
\end{align*}
\end{itemize}
\end{proposition}
\begin{proof}
Proposition \ref{prop-ref} is  a combination of \cite[Theorem 1.1, Theorem 4.1, Lemma 7.4 and Corollary 7.13]{QB3}.
\end{proof}

The following lemma will be used in the proof of Proposition \ref{prop-im}.

\begin{lemma}\label{lem-kappa}
We have
\begin{align*}
\log  \prod_{i= 1}^\ell\Big| \frac{z-p_i}{z- q_i }\Big|^2  = -  (\kappa(\mathfrak{p}, z) - \kappa(\mathfrak{q}, z)) + \mathcal{O}(1/|z|^3) \as | z| \to \infty,
\end{align*}
where the estimate $\mathcal{O}(1/|z|^3)$ is uniform as long as $p_1\cdots, p_\ell, q_1, \cdots, q_\ell$ range over a bounded subset $K\subset \C$.
\end{lemma}

\begin{proof}
Using  \eqref{close-1}, we have
\begin{align*}
& \log  \prod_{i= 1}^\ell\Big| \frac{z-p_i}{z- q_i }\Big|^2   =  \sum_{i = 1}^\ell \log \Big( 1 - \frac{p_i- q_i}{z - q_i}\Big) +  \sum_{i = 1}^\ell \log \Big( 1 - \frac{\bar{p_i}- \bar{q_i}}{\bar{z} -\bar{ q_i}}\Big)
\\
& =-  \sum_{i = 1}^\ell  \frac{p_i - q_i}{z - q_i} -  \sum_{i = 1}^\ell \frac{1}{2}   \frac{(p_i - q_i)^2}{(z - q_i)^2} -    \sum_{i = 1}^\ell  \frac{\bar{p_i} - \bar{q_i}}{\bar{z} -  \bar{q_i}}-   \sum_{i = 1}^\ell \frac{1}{2}   \frac{(\bar{p_i} - \bar{q_i})^2}{(\bar{z} - \bar{q_i})^2} + \mathcal{O}(1/|z|^3) \as |z| \to\infty.
\end{align*}
Now write
\begin{align*}
\frac{p_i - q_i}{z - q_i} +  \frac{1}{2}   \frac{(p_i - q_i)^2}{(z - q_i)^2} - \Big( \frac{p_i - q_i }{z} + \frac{p_i^2 - q_i^2}{2 z^2}\Big) & =  \frac{(p_i - q_i) q_i}{z(z-q_i)} + \frac{(p_i - q_i)^2  z^2 - (p_i^2 - q_i^2) (z-q_i)^2}{2 z^2 (z-q_i)^2}
\\
& = \frac{(p_i - q_i) P_i(z)}{ z^2 (z-q_i)^2},
\end{align*}
where $P_i$ is a polynomial of degree at most $2$. The coefficient $z^2[P_i]$ of $z^2$ in $P_i$ is given by
\[
z^2[P_i] = q_i  + \frac{p_i - q_i- p_i - q_i  }{2} =0.
\]
It follows that $\deg P_i \le 1$ and
\[
\frac{(p_i - q_i) P_i(z)}{ z^2 (z-q_i)^2} = \mathcal{O}(1/|z|^3) \text{ as $|z| \to\infty$}.
\]
Hence for any $1\le i \le \ell$, we have
\begin{align}\label{cocyle-sim}
\frac{p_i - q_i}{z - q_i} +  \frac{1}{2}   \frac{(p_i - q_i)^2}{(z - q_i)^2} = \frac{p_i - q_i }{z} + \frac{p_i^2 - q_i^2}{2 z^2} +  \mathcal{O}(1/|z|^3) \as |z| \to\infty,
\end{align}
and Lemma \ref{lem-kappa} follows.
\end{proof}

\begin{lemma}\label{lem-o3}
Recall the choice of $r_{\mathfrak{p}, \mathfrak{q}}$ in \eqref{close-1}.  We have
\[
\int\limits_{|z| \ge  r_{\mathfrak{p}, \mathfrak{q}} } \frac{1}{|z|^3} \Pi (z,z)d\lambda_\phi(z) < \infty \an \int\limits_{|z| \ge  r_{\mathfrak{p}, \mathfrak{q}} } \frac{1}{|z|^3} \Pi^{\mathfrak{q}}(z,z)d\lambda_\phi(z) < \infty.
\]
\end{lemma}

\begin{proof}
Under the assumption \eqref{sub-h} on $\phi$, we have the following Christ's estimate (see \cite[Theorem 3.1]{QB3}):
\[
\sup_{z\in\C} \Pi^{\mathfrak{q}}(z,z) e^{- 2\phi(z)} \le \sup_{z\in\C} \Pi(z,z) e^{- 2\phi(z)}  < \infty.
\]
It follows that
\begin{align*}
\int\limits_{|z| \ge  r_{\mathfrak{p}, \mathfrak{q}} } \frac{1}{|z|^3} \Pi^{\mathfrak{q}} (z,z)d\lambda_\phi(z) & \le \int\limits_{|z| \ge  r_{\mathfrak{p}, \mathfrak{q}} } \frac{1}{|z|^3} \Pi (z,z)d\lambda_\phi(z) = \int\limits_{|z| \ge  r_{\mathfrak{p}, \mathfrak{q}} } \frac{1}{|z|^3} \Pi (z,z) e^{-2 \phi(z)}d\lambda(z)
\\
&\le  \Big( \sup_{z\in\C} \Pi(z,z) e^{- 2\phi(z)} \Big) \int\limits_{|z| \ge  r_{\mathfrak{p}, \mathfrak{q}} } \frac{1}{|z|^3}d\lambda(z)  <\infty.
\end{align*}
 Lemma \ref{lem-o3}  is proved completely.
\end{proof}

\begin{proof}[Proof of Proposition \ref{prop-im}]
By Lemma \ref{lem-kappa} and Lemma \ref{lem-o3}, we have
\begin{align}
  \sum_{i =1}^\ell \int\limits_{|z|\ge r_{\mathfrak{p}, \mathfrak{q}}}   \Big|  \log\Big |  \frac{z-p_i}{ z- q_i} \Big|^2   + (  \kappa(p_i, z) - \kappa(q_i, z) )  \Big|   \Pi^{\mathfrak{q}}(z,z)    d\lambda_\phi(z) < \infty.
\end{align}
Therefore, the limit
\begin{align*}
&  \lim_{R\to\infty}
\frac{
\displaystyle{
\exp\Big(  -  \sum_{i =1}^\ell \int\limits_{|z|\le R}    \log\Big |  \frac{z-p_i}{ z- q_i} \Big|^2  \Pi^{\mathfrak{q}}(z,z)  d\lambda_\phi(z) \Big)
}
}{
\displaystyle{
\exp\Big(  \int\limits_{r_{\mathfrak{p}, \mathfrak{q}} \le |z|\le R}    (  \kappa(\mathfrak{p}, z) - \kappa(\mathfrak{q}, z) )  \Pi^{\mathfrak{q}}(z,z)  d\lambda_\phi(z) \Big)
}
}
\\
 =&  \exp\Big(   \!\!  -  \sum_{i =1}^\ell \int\limits_{|z|\le  r_{\mathfrak{p}, \mathfrak{q}}  }      \!\!  \!\!    \log\Big |  \frac{z-p_i}{ z- q_i} \Big|^2  \Pi^{\mathfrak{q}}(z,z)  d\lambda_\phi(z) \Big) \cdot  \lim_{R\to\infty} \frac{
 \displaystyle{  \exp\Big(  -  \sum_{i =1}^\ell     \int\limits_{    r_{\mathfrak{p}, \mathfrak{q}}  \le  |z|\le R}    \!\!      \log\Big |  \frac{z-p_i}{ z- q_i} \Big|^2  \Pi^{\mathfrak{q}}(z,z)  d\lambda_\phi(z) \Big)  }
 }{
 \displaystyle{
 \exp\Big(    \int\limits_{r_{\mathfrak{p}, \mathfrak{q}} \le |z|\le R}      \!\!  (  \kappa(\mathfrak{p}, z) - \kappa(\mathfrak{q}, z) )  \Pi^{\mathfrak{q}}(z,z)  d\lambda_\phi(z) \Big)
 }
 }
\end{align*}
exists and is positive.  Since $\Pi^{\mathfrak{q}}$ is a finite rank perturbation of $\Pi$, we have
\[
 \int\limits_{|z|\ge r_{\mathfrak{p}, \mathfrak{q}}}   |   \kappa(\mathfrak{p}, z) - \kappa(\mathfrak{q}, z)   |   \cdot |   \Pi^{\mathfrak{q}}(z,z)    - \Pi(z,z)|  d\lambda_\phi(z)  < \infty
\]
and hence the limit
\[
 \lim_{R\to\infty}
 \frac{
 \displaystyle{
 \exp\Big(    \int\limits_{r_{\mathfrak{p}, \mathfrak{q}} \le |z|\le R}    (  \kappa(\mathfrak{p}, z) - \kappa(\mathfrak{q}, z) )  \Pi^{\mathfrak{q}}(z,z)  d\lambda_\phi(z) \Big)
 }
 }{
 \displaystyle{
 \exp\Big(    \int\limits_{r_{\mathfrak{p}, \mathfrak{q}} \le |z|\le R}    (  \kappa(\mathfrak{p}, z) - \kappa(\mathfrak{q}, z) )  \Pi(z,z)  d\lambda_\phi(z) \Big)
 }
 }
\]
exists  and is positive. Now items (i) and (ii) of Proposition \ref{prop-im} immediately follow  from  Proposition \ref{prop-ref}.

We proceed with the proof of item (iii)  of Proposition \ref{prop-im}. To simplify notation, we let $\mu_\ell$ be  an arbitrary  fixed probability measure  in the measure class determined by $\PP_{\Pi}^{\mathfrak{p}}$ and $\PP_{\Pi}^{\mathfrak{q}}$.
 Let $K\subset \C$ be a fixed bounded subset. Let $r>0$ be chosen large enough in such a way that $r$ is larger than all $r_{\mathfrak{p}, \mathfrak{q}}$ for all $\ell$-tuples $\mathfrak{p}, \mathfrak{q}$ of distinct points in $K$. In particular, $K$ is contained in a disk $\{z\in\C: |z| \le r -\varepsilon\}$.  Denote
\[
H(R, \X; \mathfrak{p}, \mathfrak{q}): = \int\limits_{r\le |z|\le R}    (  \kappa(\mathfrak{p}, z) - \kappa(\mathfrak{q}, z) )  \Pi(z,z)  d\lambda_\phi(z)  +   \sum_{x\in\X: | x| \le R}   \log \prod_{i= 1}^\ell\Big| \frac{x-p_i}{x- q_i }\Big|^2.
\]
Then
\begin{align*}
H(R, \X; \mathfrak{p}, \mathfrak{q})  = & \underbrace{\int\limits_{r\le |z|\le R}    (  \kappa(\mathfrak{p}, z) - \kappa(\mathfrak{q}, z) )  \Pi(z,z)  d\lambda_\phi(z)  +\sum_{x\in\X: r\le | x| \le R}    (\kappa(\mathfrak{q}, z) - \kappa(\mathfrak{p}, z))}_{\text{denoted by $H_1(R, \X; \mathfrak{p}, \mathfrak{q})$}}
\\
&+   \underbrace{\sum_{x\in\X: r\le | x| \le R}   \Big( \log  \prod_{i= 1}^\ell\Big| \frac{z-p_i}{z- q_i }\Big|^2  +  \kappa(\mathfrak{p}, z) - \kappa(\mathfrak{q}, z) \Big)}_{\text{denoted by $H_2(R, \X; \mathfrak{p}, \mathfrak{q}) $}} +  \underbrace{\sum_{x\in\X:  | x| < r}   \log \prod_{i= 1}^\ell\Big| \frac{x-p_i}{x- q_i }\Big|^2}_{\text{denoted by $H_3(\X; \mathfrak{p}, \mathfrak{q}) $}}.
\end{align*}
Note that for any fixed configuration $\X$ and fixed  $R>r$, the functions $(\mathfrak{p}, \mathfrak{q}) \mapsto H_1(R, \X; \mathfrak{p}, \mathfrak{q}), (\mathfrak{p}, \mathfrak{q}) \mapsto H_2(R, \X; \mathfrak{p}, \mathfrak{q})$ and $(\mathfrak{p}, \mathfrak{q}) \mapsto H_3(\X; \mathfrak{p}, \mathfrak{q})$ are continuous on $\C^\ell \times (\C \setminus \X)^\ell$.

Fix any pair $(\mathfrak{p}, \mathfrak{q})$  of $\ell$-tuples of distinct points. Item (i) implies that  there exists a subsequence $R_n\to \infty$  such that the convergence \eqref{general-tuple} takes place for $\PP_\Pi^\mathfrak{q}$-almost every (equivalently for $\mu_\ell$-almost every) configuration $\X$. It follows that the following convergence
\begin{align}\label{contain-p-q}
H( \X; \mathfrak{p}, \mathfrak{q}) : = \lim_{n\to\infty}  H(R_n, \X; \mathfrak{p}, \mathfrak{q}) = \lim_{n\to\infty} \Big(H_1(R_n, \X; \mathfrak{p}, \mathfrak{q})  + H_2(R_n, \X; \mathfrak{p}, \mathfrak{q})\Big) + H_3( \X; \mathfrak{p}, \mathfrak{q})
\end{align}
takes place for  $\mu_\ell$-almost every configuration $\X$.

Lemma \ref{lem-o3} implies that
\[
\E_{\PP_{\Pi}^\mathfrak{q}} \Big[ \sum_{x\in\X} \frac{1}{|x|^3} \mathds{1}(|x| \ge r) \Big]< \infty
\]
and hence by Lemma \ref{lem-kappa},  the limit
\begin{align}\label{cv-r-p-q}
H_2(\X; \mathfrak{p}, \mathfrak{q})) = \lim_{n\to\infty} H_2(R_n, \X; \mathfrak{p}, \mathfrak{q})
\end{align}
exists for $\PP_\Pi^\mathfrak{q}$-almost every (equivalently for $\mu_\ell$-almost every) configuration $\X$ and moreover, for $\mu_\ell$-almost every configuration $\X$,  the limit \eqref{cv-r-p-q} converges uniformly as long as $p_1, \cdots, p_\ell, q_1, \cdots, q_\ell$ range over $K$. It follows that $(\mathfrak{p}, \mathfrak{q}) \mapsto H_3(\X; \mathfrak{p}, \mathfrak{q})$ is continuous on $K^\ell \times K^\ell$.

 The convergences \eqref{contain-p-q} and \eqref{cv-r-p-q} together imply that  the limit
\begin{align}\label{cv-H-1}
H_1(\X; \mathfrak{p}, \mathfrak{q})) = \lim_{n\to\infty} H_1(R_n, \X; \mathfrak{p}, \mathfrak{q})
\end{align}
exists  for $\mu_\ell$-almost every configuration $\X$. Note that
\begin{align}\label{formula-H}
\begin{split}
H_1(R_n, \X; \mathfrak{p}, \mathfrak{q}) =  2 \Re\Big\{ & \sum_{i = 1}^\ell    (p_i-q_i) \cdot \Big[  \int\limits_{r\le |z|\le R_n}     \frac{1}{z}    \cdot \Pi(z,z)  d\lambda_\phi(z)  - \sum_{x\in\X: r\le | x| \le R_n}  \frac{1}{x}\Big]
\\
&+ \sum_{i = 1}^\ell    (p^2_i-q^2_i) \cdot \Big[  \int\limits_{r\le |z|\le R_n}     \frac{1}{2z^2}    \cdot \Pi(z,z)  d\lambda_\phi(z)  - \sum_{x\in\X: r\le | x| \le R_n}  \frac{1}{2x^2}\Big]  \Big\}.
\end{split}
  \end{align}
  By choosing $\mathfrak{p}= (p_1, \cdots, p_\ell), \mathfrak{q} = (q_1, \cdots, q_\ell)$ in such a way that $q_1 = - p_1 = p$ and $q_i = p_i, i = 2, \cdots, \ell$, we get
  \begin{align*}
H_1(R_n, \X; \mathfrak{p}, \mathfrak{q}) =  4 \Re\Big\{  p \cdot \Big[  \int\limits_{r\le |z|\le R_n}     \frac{1}{z}    \cdot \Pi(z,z)  d\lambda_\phi(z)  - \sum_{x\in\X: r\le | x| \le R_n}  \frac{1}{x}\Big] \Big\}.
  \end{align*}
  Choosing $p = 1$ or $p = \sqrt{-1}$,  from \eqref{cv-H-1}  we obtain that the limit
  \begin{align}\label{M-1}
M_1(\X): =   \lim_{n\to\infty}   \Big[  \int\limits_{r\le |z|\le R_n}     \frac{1}{z}    \cdot \Pi(z,z)  d\lambda_\phi(z)  - \sum_{x\in\X: r\le | x| \le R_n}  \frac{1}{x}\Big]
  \end{align}
  exists  for $\mu_\ell$-almost every configuration $\X$. Consequently, using  \eqref{formula-H} and arguing as above, we conclude that  the limit
  \begin{align}\label{M-2}
 M_2(\X): =  \lim_{n\to\infty}  \Big[  \int\limits_{r\le |z|\le R_n}     \frac{1}{2z^2}    \cdot \Pi(z,z)  d\lambda_\phi(z)  - \sum_{x\in\X: r\le | x| \le R_n}  \frac{1}{2x^2}\Big]
  \end{align}
 exists  for $\mu_\ell$-almost every configuration $\X$. Hence the limit \eqref{cv-H-1} converges uniformly as long as $p_1, \cdots, p_\ell$ and $q_1, \cdots, q_\ell$ range over $K$. Moreover, we have
 \begin{align*}
  H_1(\X; \mathfrak{p}, \mathfrak{q})  = 2 \Re\Big\{ \sum_{i = 1}^\ell    (p_i-q_i)   M_1(\X)
+ \sum_{i = 1}^\ell    (p^2_i-q^2_i) M_2(\X) \Big\}
 \end{align*}
 for $\mu_\ell$-almost every configuration $\X$.
 Hence  $(\mathfrak{p}, \mathfrak{q}) \mapsto H_1(\X; \mathfrak{p}, \mathfrak{q})$ is continuous on $\C^\ell \times \C^\ell$.
 By the clear formula
 \[
 H(\X; \mathfrak{p}, \mathfrak{q})  = H_1(\X; \mathfrak{p}, \mathfrak{q})  + H_2(\X; \mathfrak{p}, \mathfrak{q})  + H_3(\X; \mathfrak{p}, \mathfrak{q}),
 \]
 we see that for $\mu_\ell$-almost every configuration $\X$,  the mapping $(\mathfrak{p}, \mathfrak{q}) \mapsto H(\X; \mathfrak{p}, \mathfrak{q})$ and hence the mapping
 \[
(\mathfrak{p}, \mathfrak{q}) \mapsto   \Psi_{\mathfrak{p}, \mathfrak{q}} (\X) = \exp ( H(\X; \mathfrak{p}, \mathfrak{q}) )
 \]
 is continuous on $K^\ell \times (K\setminus \X)^\ell$. Since $K$ is chosen arbitrarily, our functions are continuous  on $\C^\ell\times (\C\setminus \X)^\ell$ for $\mu_\ell$-almost every configuration $\X$.

 We now take $\mathcal{W}$ to be the Borel subset of $\Conf(\C)$ consisting of all configurations $\X$ such that  the limits \eqref{M-1} and \eqref{M-2} converge and
 \[
 \sum_{x\in\X} \frac{1}{|x|^3}  \mathds{1}(|x| \ge 1) < \infty.
 \]
 Obviously, $\mathcal{W}$ belongs to the tail $\sigma$-algebra.  By the argument used in the proof of item (iii), for any fixed configuration $\X\in \mathcal{W}$, the limit \eqref{general-tuple} exists and the function $(\mathfrak{p}, \mathfrak{q})\rightarrow  \Psi_{\mathfrak{p}, \mathfrak{q}}(\X)$ is continuous on $\C^\ell \times (\C\setminus \X)^{\ell}$. Hence it remains to prove that $\PP_{\Pi} (\mathcal{W}) = 1$.  For this purpose, we take any bounded Borel subset $B\subset \C$, and, using the definition of reduced Palm measure (cf. e.g., \cite[Appendix]{QB3}), write
 \begin{align}\label{r-palm}
 \int\limits_{\Conf(\C)} \sum_{x\in\X} \mathds{1}_{\mathcal{W}}(\X) \mathds{1}_B(x) d \PP_{\Pi} (\X)  = \int\limits_\C \Pi(p, p) d \lambda_\phi(p)  \int\limits_{\Conf(\C)}    \mathds{1}_{\mathcal{W}}(\X \cup \{p\}) \mathds{1}_B(p)  d \PP_{\Pi}^{p} (\X).
 \end{align}
 Since $\mathcal{W}$ belongs to the tail $\sigma$-algebra, we have $\mathds{1}_{\mathcal{W}}(\X \cup \{p\})  =  \mathds{1}_{\mathcal{W}}(\X)$. Moreover, by the proof of item (iii) above, we have $\PP^{p}_{\Pi}(\mathcal{W})=1$.  Hence \eqref{r-palm} can be re-written as
 \begin{align*}
  \int\limits_{\mathcal{W}}  \#_B(\X) d \PP_{\Pi} (\X) &  = \int\limits_B \Pi(p, p) d \lambda_\phi(p)  \int\limits_{\Conf(\C)}    \mathds{1}_{\mathcal{W}}(\X)   d \PP_{\Pi}^{p} (\X) 
  \\
  & =  \int\limits_B \Pi(p, p) d \lambda_\phi(p) =   \int\limits_{\Conf(\C)}  \#_B(\X) d \PP_{\Pi} (\X).
 \end{align*}
 Since $B\subset \C$ is arbitrary, the above equality implies that $\PP_{\Pi}(\mathcal{W})=1$.

Item (iii) is proved completely.
 The proof of item (iv) is postponed to  Subsection \ref{subsec-iv}.
  \end{proof}

%%%%%%%%%%%%%%%%%%%%%%%%%%%%%%%%%%%%%%%
%%%%%%%%%%%%%%%%%%%%%%%%%%%%%%%%%%%%%%%
%%%%%%%%%%%%%%%%%%%%%%%%%%%%%%%%%%%%%%%
%%%%%%%%%%%%%%%%%%%%%%%%%%%%%%%%%%%%%%%
%%%%%%%%%%%%%%%%%%%%%%%%%%%%%%%%%%%%%%%
%%%%%%%%%%%%%%%%%%%%%%%%%%%%%%%%%%%%%%%
%%%%%%%%%%%%%%%%%%%%%%%%%%%%%%%%%%%%%%%
%%%%%%%%%%%%%%%%%%%%%%%%%%%%%%%%%%%%%%%
%%%%%%%%%%%%%%%%%%%%%%%%%%%%%%%%%%%%%%%
%%%%%%%%%%%%%%%%%%%%%%%%%%%%%%%%%%%%%%%
%%%%%%%%%%%%%%%%%%%%%%%%%%%%%%%%%%%%%%%
%%%%%%%%%%%%%%%%%%%%%%%%%%%%%%%%%%%%%%%
%%%%%%%%%%%%%%%%%%%%%%%%%%%%%%%%%%%%%%%
%%%%%%%%%%%%%%%%%%%%%%%%%%%%%%%%%%%%%%%
%%%%%%%%%%%%%%%%%%%%%%%%%%%%%%%%%%%%%%%
%%%%%%%%%%%%%%%%%%%%%%%%%%%%%%%%%%%%%%%

\section{Computation of normalization constant in the radial case}

This section is devoted to the proof of Theorem \ref{thm-main1}.

\subsection{Finite dimensional approximations}

From now on, we fix two $\ell$-tuples $\mathfrak{p} = (p_1, \dots, p_\ell)$ and $\mathfrak{q} = (q_1, \dots, q_\ell)$ of distinct points in $\C$.   Since $\phi$ is radial,   we have
\begin{align}\label{taylor-kernel}
\Pi(z,w) =  \sum_{k=0}^\infty   a_k^2 (z\bar{w})^k, \text{ where } a_k = \frac{1}{\| z^k\|_{L^2(\C, d\lambda_\phi)}}.
\end{align}
Natural finite-dimensional approximations of $\Pi$ are  given by
\begin{align}\label{finite-app}
\Pi_n(z,w) =  \sum_{k=0}^{n-1} a_k^2 (z\bar{w})^k.
\end{align}
For any $n\ge\ell$ we then set
\[
\Pi_n^{\mathfrak{q}}:  = (\Pi_n)^{\mathfrak{q}}
\]
 and obtain natural finite-dimensional approximations $\Pi_n^{\mathfrak{q}}$ of $\Pi^{\mathfrak{q}}$.
 Our aim now is to show the left-hand side of \eqref{com-exp} can indeed be computed  by approximation.

\subsection{Convergence of finite-dimensional approximations}
In this subsection, Theorem \ref{thm-main1} is reduced to Lemma \ref{lem-trace-class}, Proposition \ref{prop-finite} and Proposition \ref{prop-order}.

\begin{notation}\label{many-notation}
Recall that $r_{\mathfrak{p}, \mathfrak{q}}>0$ is chosen in such a way that \eqref{close-1} holds.  Let $r>0, R>0$ be any two positive numbers such that  $R> r>r_{\mathfrak{p}, \mathfrak{q}}$.

(i) We denote
\begin{align}\label{notation-trunc}
\begin{split}
\chi_{r}^\infty = \ch (| z| \ge r) ; \,   \chi_{0}^r = \ch (0\le | z| \le r); &  \, \chi_{R}^\infty = \ch (| z| \ge R);   \,  \chi_{0}^R = \ch (0\le | z| \le R); \, \chi_{r}^R = \ch (r \le | z| \le R);
\\
g(z) = \Big |  \frac{z-p_i}{ z- q_i} \Big|^2 ; &\quad  h(z)  = \prod_{i=1}^\ell \Big |  \frac{z-p_i}{ z- q_i} \Big|^2  - 1;
\\
h_r^\infty = h \chi_{r}^\infty; \quad h_0^r = h \chi_{0}^r; &\quad h_R^\infty = h \chi_{R}^\infty; \quad h_0^R = h \chi_{0}^R; \quad h_r^R = h\chi_r^R.
\end{split}
\end{align}
(ii) For any $n \ge \ell$, we denote
 \begin{align}\label{notation-T}
 \begin{split}
 T  = \sgn(h)  \sqrt{| h|}    \Pi^{\mathfrak{q}}  \sqrt{ |h|}; &\quad  T_n  = \sgn(h) \sqrt{| h|}  \Pi_n^{\mathfrak{q}}  \sqrt{|h|};
 \\
T_R =\chi_0^R  T\chi_0^R =   \sgn( h_0^R) \sqrt{|  h_0^R|}   \Pi^{\mathfrak{q}}  \sqrt{|  h_0^R|} ; &\quad T_{n, R} = \chi_0^R  T_{n}  \chi_0^R
 = \sgn( h_0^R) \sqrt{|  h_0^R|}   \Pi_n^{\mathfrak{q}}  \sqrt{|  h_0^R|}.
 \end{split}
\end{align}
\end{notation}

Our first lemma, proved in \S \ref{sec-prop-finite} below, shows that our approximating operators belong to the trace class:
\begin{lemma}\label{lem-trace-class}
The operators $T_{n, R}, T_R$ and $T_n$ are all  trace class. Moreover,  $T_{n, R}$ converges to $T_R$   as $n\to\infty$ and $T_{n, R}$ converges to $T_n$ as $R\to\infty$,  both convergences taking place in the space of trace class operators.
\end{lemma}

We next compute the limits of the expectations of our multiplicative functionals. The formulas are related
 to 
(4.3) in Osada-Shirai \cite{Osada-Shirai}. 
For brevity, we write
\[
\Delta(\mathfrak{p}) =  \prod_{1  \leqslant i< j \leqslant \ell} ( p_j- p_i).
\]
\begin{proposition}\label{prop-finite}
For any $n \ge \ell$, we have the following equality:
\begin{align}\label{comp-mul}
\lim_{R \to\infty}  \E_{\Pi_n^{\mathfrak{q}}}  \Big[\prod_{x \in \mathscr{X}: | x| \le R} \prod_{i= 1}^\ell \Big| \frac{x-p_i}{x-q_i}\Big|^2\Big]  =   \frac{ \Det_{i, j =1}^\ell (\Pi_n(p_i, p_j))}{ \Det_{i, j =1}^\ell(\Pi_n(q_i, q_j))}   \frac{| \Delta(\mathfrak{q})|^2}{ | \Delta(\mathfrak{p})|^2}.
\end{align}
Consequently, we have
\begin{align}\label{finite-limit}
\lim_{n\to\infty}\lim_{R \to\infty}  \E_{\Pi_n^{\mathfrak{q}}}  \Big[\prod_{x \in \mathscr{X}: | x| \le R} \prod_{i= 1}^\ell \Big| \frac{x-p_i}{x-q_i}\Big|^2\Big] =\frac{ \Det_{i, j =1}^\ell (\Pi(p_i, p_j))}{ \Det_{i, j =1}^\ell(\Pi(q_i, q_j))}   \frac{| \Delta(\mathfrak{q})|^2}{ | \Delta(\mathfrak{p})|^2}.
\end{align}
\end{proposition}

\begin{proposition}\label{prop-order}
The order of limits in \eqref{finite-limit} is immaterial, that is,
\begin{align*}
 \lim_{R \to\infty}  \lim_{n\to\infty}  \E_{\Pi_n^{\mathfrak{q}}}  \Big[\prod_{x \in \mathscr{X}: | x| \le R} \prod_{i= 1}^\ell \Big| \frac{x-p_i}{x-q_i}\Big|^2\Big]  = \lim_{n\to\infty}\lim_{R \to\infty}  \E_{\Pi_n^{\mathfrak{q}}}  \Big[\prod_{x \in \mathscr{X}: | x| \le R} \prod_{i= 1}^\ell \Big| \frac{x-p_i}{x-q_i}\Big|^2\Big].
\end{align*}
\end{proposition}

Proposition \ref{prop-finite} and Proposition \ref{prop-order} will be proved in \S \ref{sec-prop-finite} and in \S \ref{sec-prop-order}.

\begin{proof}[Proof of Theorem \ref{thm-main1}]
We first claim that for any fixed $R>0$,
\begin{align}\label{exp-conv}
 \lim_{n\to\infty}  \E_{\Pi_n^{\mathfrak{q}}}  \Big[\prod_{x \in \mathscr{X}: | x| \le R} \prod_{i= 1}^\ell \Big| \frac{x-p_i}{x-q_i}\Big|^2\Big] =   \E_{\Pi^{\mathfrak{q}}}  \Big[\prod_{x \in \mathscr{X}: | x| \le R} \prod_{i= 1}^\ell \Big| \frac{x-p_i}{x-q_i}\Big|^2\Big].
\end{align}
Since $T_{n, R}$ and $T_R$ are both in trace class, the expectations of multiplicative functionals are given by corresponding Fredholm determinants:
\begin{align}\label{}
 \E_{\Pi_n^{\mathfrak{q}}}  \Big[\prod_{x \in \mathscr{X}: | x| \le R} \prod_{i= 1}^\ell \Big| \frac{x-p_i}{x-q_i}\Big|^2\Big]  = \det(1 +T_{n, R}); \quad  \E_{\Pi^{\mathfrak{q}}}  \Big[\prod_{x \in \mathscr{X}: | x| \le R} \prod_{i= 1}^\ell \Big| \frac{x-p_i}{x-q_i}\Big|^2\Big]  = \det(1 +T_{ R}).
\end{align}
Now the convergence \eqref{exp-conv} follows immediately from Lemma \ref{lem-trace-class}.  Applying Corollary \ref{cor-radial}, we obtain
\begin{align}\label{R-n}
\E_{\PP_{\Pi}^{\mathfrak{q}}} [\Gamma_{\mathfrak{p}, \mathfrak{q}}]  = \lim_{R\to\infty} \E_{\Pi^{\mathfrak{q}}}  \Big[\prod_{x \in \mathscr{X}: | x| \le R} \prod_{i= 1}^\ell \Big| \frac{x-p_i}{x-q_i}\Big|^2\Big] =  \lim_{R \to\infty}  \lim_{n\to\infty}  \E_{\Pi_n^{\mathfrak{q}}}  \Big[\prod_{x \in \mathscr{X}: | x| \le R} \prod_{i= 1}^\ell \Big| \frac{x-p_i}{x-q_i}\Big|^2\Big].
\end{align}
An application of Proposition \ref{prop-finite} and Proposition \ref{prop-order} yields the desired result \eqref{com-exp}.
 \end{proof}

\subsection{Proof of Lemma \ref{lem-trace-class} and Proposition \ref{prop-finite}}\label{sec-prop-finite}

Let  $\mathscr{H}$ denote the Hilbert space $L^2(\C,  d \lambda_\phi)$. Write $\|\cdot\|_{\mathscr{H}}$ for the norm in $L^2(\C, d \lambda_\phi)$.   For any real number $s \in [1, \infty)$, let  $\mathscr{C}_s(\mathscr{H})$ denote the von~Neumann-Schatten $s$-class on $\mathscr{H}$, that is,  the class of bounded linear operators $A$ on $\mathscr{H}$, such that $\tr(| A|^s) <\infty$, where $|A|^s: = (A^* A)^{\frac{s}{2}}$ is the continuous functional calculus of the self-adjoint positive operator $A^*A$ under the function $t \mapsto t^{\frac{s}{2}}$ defined on $[0, \infty)$.  In particular, the von~Neumann-Schatten $1$-class coincides with the trace class while the von~Neumann-Schatten $2$-class coincides with the Hilbert-Schmidt class.  The space $\mathscr{C}_s(\mathscr{H})$ can be equipped  with the norm $\| \cdot \|_s$ defined by $\| A \|_s = \tr(|A|^s)^{1/s}$. In particular, the norm $\|\cdot\|_2$ coincides with the Hilbert-Schmidt norm, which we denote also by $\| \cdot \|_{HS}$. The von~Neumann-Schatten $s$-class norm $\|\cdot\|_s$ has the following  properties:
\begin{align}\label{op-ideal}
\|A\|_s = \| A^*\|_s \an \| BAC\|_{s} \le \|B\| \cdot \|A\|_s\cdot \|C\|,
\end{align}
where $\|\cdot\|$ stands for the usual operator norm.
The following H\"older inequalities for operators in von~Neumann-Schatten $s$-classes will be frequently used:
\begin{align}\label{S-holder}
\|AB\|_{s} \le \|A\|_{s_0}^{\theta} \cdot \|B\|_{s_1}^{1-\theta}, \text{ if $\frac{1}{s} = \frac{\theta}{s_0} + \frac{1-\theta}{s_1}, \theta\in(0, 1)$.}
\end{align}

We will need the following standard proposition. Recall that for two positive operators $A, B$ on $\mathscr{H}$, we write $A\le B$ if $A-B$ is a positive operator. In particular, if $A, B$ are both orthogonal projections on $\mathscr{H}$, then $A\le B$ means that the range of $A$ is contained in the range of $B$.

\begin{proposition}\label{prop-cp}
Let $s\in [1, \infty)$ and let $A\in \mathscr{C}_s(\mathscr{H})$. Suppose that $P$ is an orthogonal projection  on $\mathscr{H}$ and $P_n$'s orthogonal projection on $\mathscr{H}$ such that $P_n\le P$ and $P_n \le P_{n+1}$ for any $n\in\N$. If the sequence $(P_n)_{n\in\N}$ converges  to $P$ in the strong operator topology,  then
\begin{align}\label{cv-p}
\lim_{n\to\infty} \| AP - AP_n\|_s=0.
\end{align}
\end{proposition}

\begin{proof}
 Let us first show that
\begin{align}\label{norm-cv}
\lim_{n\to\infty} \| AP - AP_n\|=0,
\end{align}
where $\| \cdot \|$ is the operator norm. Indeed, if \eqref{norm-cv} does not hold, then by passing to a subsequence if necessary, we may assume that there exists $\varepsilon>0$, such that $\| AP - AP_n\| > \varepsilon$, for any $n\in\N$.  Obviously,  $AP_k$ converges in strong operator topology to $AP$ as $k\to\infty$,  hence  $AP_k- AP_n$ converges  to in strong operator topology to $AP - AP_n$ as $k\to\infty$ for any $n\in\N$. It follows that for any $n\in\N$,
\[
\varepsilon < \|AP-AP_n\| \le \liminf_{k \to\infty}\|AP_k-AP_n\|.
\]
Consequently, we can find a subsequence $n_1 < n_2 < \cdots$ of positive integers, such that
\[
\|AP_{n_{i+1}} - AP_{n_i}\| > \varepsilon, \text{\, for any $ i = 1, 2, \cdots.$}
\]
Then  for any $i\in\N$ we can find a unit vector $\xi_i$ in the range $\Ran(P_{n_{i+1}} - P_{n_i})$ of the projection $P_{n_{i+1}} - P_{n_i}$ such that $\|A \xi_i\|_{\mathscr{H}} > \varepsilon$. Note that by construction,   $\xi_i$ converges weakly in $\mathscr{H}$ to $0$ as $i$ goes to infinity. Using the compactness of the operator $A$, we get $\lim_{i\to\infty}\| A\xi_i\|_{\mathscr{H}}=0$. This contradiction implies that we must have \eqref{norm-cv}.

Now by applying \cite[Theorem 2.17]{Simon-trace}, we get the desired convergence \eqref{cv-p}.
\end{proof}

\begin{proof}[Proof of Lemma \ref{lem-trace-class}]
Note that $T_n$ and $T_{n, R}$ are both finite rank bounded linear operators, so they are in trace class.  Now we prove that $T_R$ is in trace class.
Since there exists $C>0$ such that
\[
\Pi^{\mathfrak{q}}(z,z) \le C \prod_{i=1}^\ell | z-q_i|^2 \quad \text{ if $|z|\le R$},
\]
the function $|  h(z)| \Pi^{\mathfrak{q}}(z, z) $ is bounded on the disk $\{z \in \C: | z| \le R\}$. Therefore, 
\[
 \int_{|z|\le R}  |  h(z)| \Pi^{\mathfrak{q}}(z, z) d\lambda_\phi(z)< \infty.
\]
By using the inequality 
\[
\int_{|w|\le R}  | \Pi^{\mathfrak{q}}(z, w) |^2 d\lambda_\phi(w) \le  \int_{\C}  | \Pi^{\mathfrak{q}}(z, w) |^2 d\lambda_\phi(w)  =  \Pi^{\mathfrak{q}}(z, z), 
\]
where the last equality is a consequence of the fact that $ \Pi^{\mathfrak{q}}$ is an orthogonal projection, we immediately obtain 
\[
 \|\sqrt{|h_0^R|} \Pi^{\mathfrak{q}}  \|_{HS}^2 = \int_{|z|\le R} \int_{|w|\le R}  |  h(z)| | \Pi^{\mathfrak{q}}(z, w) |^2 d\lambda_\phi(z) d\lambda_\phi(w)\le  \int_{|z|\le R}  |  h(z)| \Pi^{\mathfrak{q}}(z, z) d\lambda_\phi(z)< \infty.
\]
It follows that $\sqrt{|h_0^R|} \Pi^{\mathfrak{q}}$ is Hilbert-Schmidt, hence  $\sqrt{|h_0^R|} \Pi^{\mathfrak{q}} \sqrt{|h_0^R|}$ and  $T_R$ are both in trace class. The assertions concerning convergences in $\mathscr{C}_1(\mathscr{H})$ are immediate from Proposition \ref{prop-cp}.
\end{proof}

\begin{proof}[Proof of Proposition \ref{prop-finite}]
By Lemma \ref{lem-trace-class}, $T_{n, R}$ converges to $T_n$ in $\mathscr{C}_1(\mathscr{H})$ as $R\to\infty$, whence
\[
\lim_{R\to\infty} \Det(1 +   T_{n, R} ) = \Det(1 + T_n),
\]
or, in other words,
\begin{align*}
\lim_{R\to\infty} \E_{\Pi_n^{\mathfrak{q}}}  \Big[\prod_{x \in \mathscr{X}: | x| \le R} \prod_{i= 1}^\ell \Big| \frac{x-p_i}{x-q_i}\Big|^2\Big]   = \E_{\Pi_n^{\mathfrak{q}}}  \Big[\prod_{x \in \mathscr{X}} \prod_{i= 1}^\ell \Big| \frac{x-p_i}{x-q_i}\Big|^2\Big] .
\end{align*}
Recall that $\PP_{\Pi_n}$ is an orthogonal polynomial ensemble given by following probability measure on $\C^n$:
\[
\frac{1}{Z_n(\phi)} \cdot \prod_{1\leqslant i < j \leqslant n } | z_i - z_j|^2 \cdot  \prod_{j=1}^n  d\lambda_\phi(z_i).
\]
 The reduced Palm measure $\PP_{\Pi_n}^{\mathfrak{q}} = \PP_{\Pi_n^{\mathfrak{q}}}$ is also an orthogonal polynomial ensemble, given by the following probability measure  on $\C^{n-\ell}$:
 \[
 \frac{1}{Z_n(\phi, \mathfrak{q})}   \cdot \prod_{1\leqslant i < j \leqslant n-\ell} | z_i - z_j|^2 \cdot  \prod_{j=1}^{n-\ell}\Big( \prod_{k=1}^\ell | z_j - q_k |^2   d\lambda_\phi(z_j) \Big).
\]
By definition, the normalization constant $Z_n(\phi, \mathfrak{q})$ is given by the formula
\begin{align}\label{norm-const}
Z_n(\phi, \mathfrak{q}) = \int_{\C^{n-\ell}}\prod_{1\leqslant i < j \leqslant n-\ell} | z_i - z_j|^2 \cdot  \prod_{j=1}^{n-\ell}\Big( \prod_{k=1}^\ell | z_j - q_k |^2   d\lambda_\phi(z_j) \Big).
\end{align}
By the definition of $\ell$-th order correlation function of $\PP_{\Pi_n}$ (see, e.g.,  \cite[formula (2.3)]{Soshnikov}), we have
\begin{align}\label{corr-help}
\Det_{i, j =1}^\ell (\Pi_n(q_i, q_j)) =   \frac{n!}{ (n-\ell)!} \int_{\C^{n-\ell}} \frac{1}{Z_n(\phi)}\prod_{1\leqslant i < j \leqslant n- \ell } | z_i - z_j|^2     | \Delta(\mathfrak{q})|^2    \prod_{j=1}^{n-\ell}   \prod_{k=1}^\ell | z_j - q_k |^2       \prod_{j=1}^{n-\ell} d\lambda_\phi(z_i).
\end{align}
It follows that
\begin{align}\label{exp-norm}
Z_n(\phi, \mathfrak{q})  = Z_n(\phi)   \frac{ (n-\ell)!}{n!}   \frac{ \Det_{i, j =1}^\ell (\Pi_n(q_i, q_j))}{|  \Delta(\mathfrak{q})|^2}.
\end{align}
On the other hand, we also have
\begin{align*}
 \E_{\Pi_n^{\mathfrak{q}}}  \Big[\prod_{x \in \mathscr{X}} \prod_{i= 1}^\ell \Big| \frac{x-p_i}{x-q_i}\Big|^2\Big] &=    \int_{\C^n}    \Big[\prod_{ j=1}^{n-\ell} \prod_{k= 1}^\ell \Big| \frac{z_j-p_k}{z_j-q_k}\Big|^2\Big]   \frac{1}{Z_n(\phi, \mathfrak{q})}    \prod_{1 \leqslant i < j \leqslant n-\ell} | z_i - z_j|^2 \prod_{j=1}^{n-\ell}\Big( \prod_{k=1}^\ell | z_j - q_k |^2   d\lambda_\phi(z_j) \Big)
 \\
 & = \int_{\C^n}   \frac{1}{Z_n(\phi, \mathfrak{q})}   \prod_{1\leqslant i < j \leqslant n-\ell} | z_i - z_j|^2  \prod_{j=1}^{n-\ell}\Big( \prod_{k=1}^\ell | z_j - p_k |^2   d\lambda_\phi(z_j) \Big).
\end{align*}
Consequently, by using defining normalization constant $Z_n(\phi, \mathfrak{p})$ as in \eqref{norm-const}, we get
\[
\E_{\Pi_n^{\mathfrak{q}}}  \Big[\prod_{x \in \mathscr{X}} \prod_{i= 1}^\ell \Big| \frac{x-p_i}{x-q_i}\Big|^2\Big] =  \frac{ Z_n(\phi, \mathfrak{p})}{Z_n(\phi, \mathfrak{q})}.
\]
Now by applying the formula \eqref{exp-norm} for  $Z_n(\phi, \mathfrak{p})$ and  $Z_n(\phi, \mathfrak{q})$, we arrive at the desired equality \eqref{comp-mul}.
\end{proof}

\subsection{Proof of Proposition \ref{prop-order}}\label{sec-prop-order}

The proof of Proposition \ref{prop-order} is quite involved.   Technical difficulties arise since the function
\begin{align}\label{scheme-h}
h(z) = \prod_{i= 1}^\ell\Big| \frac{z-p_i}{z- q_i }\Big|^2 - 1,
\end{align}
that we used for computing  $\Gamma_{\mathfrak{p}, \mathfrak{q}}$  and hence the Radon-Nikodym derivative $\frac{d\PP_\Pi^{\mathfrak{p}}}{d\PP_{\Pi}^{\mathfrak{q}}}$, has poles and decays at infinity quite slowly. The  key point  is the  factorization \eqref{exp-decom}.
Our argument can be summarized as follows.

\noindent \textbf{Step 1.} The expectations, with respect to the determinantal point processes $\PP_{\Pi_n}^{\mathfrak{q}}$,  of the  multiplicative functionals
\[
\prod_{x \in \mathscr{X}: | x| \le R} \prod_{i= 1}^\ell \Big| \frac{x-p_i}{x-q_i}\Big|^2
\]
are given by Fredholm determinants $\det(1+T_{n, R})$, where $T_{n, R}$ is defined in \eqref{notation-T}. Although the operators $T_{n, R}$ are in trace class  for any $n\in\N$ and $R>0$,  the limits
\begin{align*}
\lim_{n\to\infty} \lim_{R\to\infty} T_{n, R} \an \lim_{R\to\infty} \lim_{n\to\infty}  T_{n, R}
\end{align*}
do not exist in the space $\mathscr{C}_1(\mathscr{H})$ of trace class operators. These limits do however exist in the space $\mathscr{C}_3(\mathscr{H})$, the von~Neumann-Schatten $3$-class,  and are both equal to $T$ (defined in \eqref{notation-T}), see Lemma \ref{lem-S3} and Lemma \ref{prop-order-det}.

\bigskip

\noindent \textbf{Step 2.} We represent the Fredholm determinant $\det(1+T_{n, R})$ as a product of  the regularized Fredholm determinant and  the regularization factor
\[
\det(1+T_{n, R}) = \underbrace{\Det_3(1 + T_{n, R})}_{\text{regular part}} \cdot  \underbrace{\exp\Big(\tr(T_{n,R}) - \frac{1}{2} \tr(T_{n, R}^2)\Big)}_{\text{regularization part}}.
\]
see Proposition \ref{prop-factor}.
The definition of the regularized Fredholm determinant $\Det_3$ is recalled in \S \ref{subsec-factor}.

\bigskip

\noindent \textbf{Step 3.} We further  factorize the regularization part $\exp(\tr(T_{n,R}) - \frac{1}{2} \tr(T_{n, R}^2))$  or, equivalently, we decompose the  integral
\begin{align}\label{int-sum-dec}
I(n, R): = \tr(T_{n,R}) - \frac{1}{2} \tr(T_{n, R}^2)  = \int_{\C} T_{n, R}(z,z) d\lambda_\phi(z) - \frac{1}{2} \int_{\C} [T_{n, R}^2](z,z)  \lambda_\phi(z)
\end{align}
 into summands controlling, respectively, the contribution of the neighbourhood of the poles of the function $h(z)$ (defined in \eqref{scheme-h}), the main contribution   and the contribution at infinity. It is then much easier to control these summands separately. The contribution of the neighbourhoods of poles is controlled in Lemma \ref{lem-E1}, the main part is controlled in Lemmata \ref{lem-E2}, \ref{lem-E3} and the contribution at infinity  is controlled in Lemma \ref{lem-E4}.

 \subsubsection{The factorization formula}\label{subsec-factor}
 For stating the factorization formula \eqref{exp-decom},  let  us  first briefly  recall necessary material from the theory of regularized Fredholm determinants (see, e.g. Helemskii \cite{helem}, Simon \cite{Simon-det}), which will be a crucial ingredient in this section.

For any $n\in\N$, the regularized Fredholm determinant $\Det_n$ is defined as follows. If $A \in \mathscr{C}_1(\mathscr{H})$, then we define
 \begin{align}\label{f-reg-fred}
 \Det_n(1 + A)  = \det(1+ A) \cdot \exp\Big( \sum_{k=1}^{n-1} \frac{(-1)^k }{k} \tr(A^k)\Big),
 \end{align}
 where $\det(1 + A)$ is classical Fredholm determinant. The map $A \mapsto \Det_n(1+A)$ is continuous in the $\|\cdot\|_n$-norm. Consequently, since $\mathscr{C}_1(\mathscr{H})$ is a dense subspace in $\mathscr{C}_n(\mathscr{H})$, the map  $A\mapsto \Det_n(1+A)$ defined by the formula \eqref{f-reg-fred}  is uniquely continuously extended onto $\mathscr{C}_n(\mathscr{H})$.

Theorem 6.5 in Simon \cite{Simon-det} states that for any $n\in\N$, there exists  $\gamma_n> 0$ such that for any $A, B \in \mathscr{C}_n(\mathscr{H})$, we have
\begin{align}\label{cont-det}
| \Det_n (1 + A)  - \Det_n(1 + B) | \le \| A-B\|_n \exp[\gamma_n (\| A\|_n + \| B\|_n + 1)^n].
\end{align}

\begin{proposition}[Factorization]\label{prop-factor}
For any $n \ge \ell$, we have
\begin{align}\label{exp-decom}
 & \E_{\PP_{\Pi_n}^{\mathfrak{q}}}  \Big[\prod_{x \in \mathscr{X}: | x| \le R} \prod_{i= 1}^\ell \Big| \frac{x-p_i}{x-q_i}\Big|^2\Big]  =  \Det(1 + T_{n, R})= \Det_3(1 + T_{n, R}) \cdot \exp\Big( \sum_{i=1}^4 E_i(n, R; r) \Big),
\end{align}
with $E_1(n, R; r), E_2(n, R; r),  E_3(n, R; r), E_4(n, R; r)$  given by
\begin{align}\label{E-4}
\begin{split}
E_1(n, R; r) &=    \tr(\chi_0^r T_{n, R} ) - \frac{1}{2} \tr(\chi_0^r T_{n, R}^2)     - \frac{1}{2}  \tr(h_r^R   \Pi_n^{\mathfrak{q}}    h_0^r \Pi_n^{\mathfrak{q}}  ) ;
\\
E_2(n, R; r) &= \int_{ r \le | z | \le R}    \Big( h(z) - \frac{h(z)^2}{2}   + (\kappa(\mathfrak{p}, z) - \kappa(\mathfrak{q}, z)) \Big) \Pi_n(z,z)   d\lambda_\phi(z) ;
\\
E_3(n, R; r) &= \int_{ r \le | z | \le R}    \Big(  h(z) - \frac{h(z)^2}{2} \Big) (\Pi_n^{\mathfrak{q}}(z,z)    - \Pi_n(z,z)) d\lambda_\phi(z) ;
\\
E_4(n, R; r) & =  \frac{1}{4}\|  [ h_r^R,  \Pi_n^{\mathfrak{q}}] \|_{HS}^2; \text{\, where $ [ h_r^R,  \Pi_n^{\mathfrak{q}}] =   h_r^R \Pi_n^{\mathfrak{q}} -  \Pi_n^{\mathfrak{q}}h_r^R$ is the commutator of $ h_r^R$ and $\Pi_n^{\mathfrak{q}}$.}
\end{split}
\end{align}
\end{proposition}

{\flushleft \bf{Remark.} }
Using the notation \eqref{int-sum-dec}, we can write
\begin{align}\label{F-n-R}
I(n, R)=  \sum_{i=1}^4 E_i (n, R; r).
\end{align}

Let us explain more precisely the meaning of these terms $E_i(n, R; r)$ as the decomposition summands  of the integral \eqref{int-sum-dec}.  The term $E_1(n, R; r)$ corresponds to the contribution of the neighbourhood of the poles of the function $h(z)$; the terms $E_2(n, R; r)$ and $E_3(n, R; r)$ together correspond to the the main contribution and the term  $E_4(n, R; r)$ corresponds to the contribution  at infinity. The estimate of $E_4(n, R; r)$, the contribution at infinity, will use in a crucial way  the radial assumption of the function $\phi$ and hence the radially-symmetric property of the kernel $\Pi$ and also its finite approximations $\Pi_n$.
Note that  for fixed $n\ge \ell$ and $R > r_{\mathfrak{p}, \mathfrak{q}}$,   we have a family of  decompositions \eqref{exp-decom} indexed by a real number $r$ that ranges in the open interval $(r_{\mathfrak{p}, \mathfrak{q}}, R)$.

\begin{proof}[Proof of Proposition \ref{prop-factor}]
We claim that for any $n \ge \ell$ and any $R> r$, we have
\begin{align}\label{trace-square}
\tr(( h_r^R  \Pi_n^{\mathfrak{q}})^2)  =   \int_{r \le  | z| \le R}  h  (z)^2 \Pi_n^{\mathfrak{q}} (z,z) d\lambda_\phi(z) -  \frac{1}{2}\| [ \Pi_n^{\mathfrak{q}},   \, h_r^R ] \|_{HS}^2 .
\end{align}
Indeed, we may write
\begin{align}\label{c-s-com}
 \| [ \Pi_n^{\mathfrak{q}},\,   h_r^R ] \|_{HS}^2 = \| \Pi_n^{\mathfrak{q}}     h_r^R \|_{HS}^2    + \|    h_r^R  \Pi_n^{\mathfrak{q}}  \|_{HS}^2 - 2 \Re ( \tr(  h_r^R  \Pi_n^{\mathfrak{q}} \cdot  h_r^R  \Pi_n^{\mathfrak{q}} )).
\end{align}
Observe that for any bounded real function $f$ on a measure space $(\Omega, \mu)$ and any finite rank orthogonal projection $P$ on $L^2(\Omega, \mu)$, we have $\tr(f P fP) \in \R$. Indeed, by using the identity $f = f\chi_{f\ge 0} - f\chi_{f< 0}$, it suffices to show that if $f_1, f_2$ are two non-negative bounded functions, then $\tr(f_1 P f_2 P) \in\R$. But this follows from the clear equality $\tr(f_1 P f_2 P) = \tr(f_1^{1/2} P f_2 P f_1^{1/2}) \ge 0$.  Now \eqref{c-s-com} can be written as
\begin{align}\label{exp-com}
 \| [ \Pi_n^{\mathfrak{q}},   h_r^R ] \|_{HS}^2  = 2 \|    h_r^R \Pi_n^{\mathfrak{q}}  \|_{HS}^2 - 2  \tr(  h_r^R  \Pi_n^{\mathfrak{q}} \cdot  h_r^R  \Pi_n^{\mathfrak{q}} )=  2 \int_{r \le  | z| \le R}  h  (z)^2 \Pi_n^{\mathfrak{q}} (z,z) d\lambda_\phi(z) - 2 \tr(( h_r^R  \Pi_n^{\mathfrak{q}})^2),
\end{align}
and \eqref{trace-square} follows.

Now since $T_{n, R}\in \mathscr{C}_1(\mathscr{H})$, the expectation of the corresponding multiplicative functional with respect to  $\PP_{\Pi_n}^{\mathfrak{q}}$ is given by Fredholm determinant:
\[
 \E_{\PP_{\Pi_n}^{\mathfrak{q}}} \Big[\prod_{x \in \mathscr{X}: | x| \le R} \prod_{i= 1}^\ell \Big| \frac{x-p_i}{x-q_i}\Big|^2\Big] = \Det(1 + T_{n, R}) = \Det_3(1 + T_{n, R})  \exp\Big( \tr (T_{n, R}) - \frac{1}{2} \tr (T_{n, R}^2) \Big).
\]
To prove Proposition \ref{prop-factor},  it suffices to prove that
\begin{align}\label{4-terms}
 \tr (T_{n, R}) - \frac{1}{2} \tr (T_{n, R}^2) = \sum_{i=1}^4 E_i(n, R; r) .
\end{align}
To this end, we first write $1 = \chi_0^r + \chi_r^\infty$, whence
\begin{align}\label{seperate-r}
&  \tr (T_{n, R}) - \frac{1}{2} \tr (T_{n, R}^2) =  \tr(\chi_0^r T_{n, R} ) - \frac{1}{2} \tr(\chi_0^r T_{n, R}^2) +  \tr(\chi_r^\infty T_{n, R} ) - \frac{1}{2} \tr(\chi_r^\infty T_{n, R}^2).
\end{align}
Note the clear equality
\begin{align}\label{trace-away}
\begin{split}
&  \tr(\chi_r^\infty T_{n, R} ) = \int_{r \le | z| \le R } h(z)  \Pi_n^{\mathfrak{q}}(z,z) d\lambda_\phi(z)
  \\
  & = \int_{r \le | z| \le R } h(z)  \Pi_n(z,z) d\lambda_\phi(z)  +\int_{r \le | z| \le R } h(z)  (\Pi_n^{\mathfrak{q}}(z,z)-  \Pi_n(z,z) )  d\lambda_\phi(z).
  \end{split}
\end{align}
Since $\tr(AB)=  \tr(BA)$, we obtain
\begin{align*}
\tr(\chi_r^\infty T_{n, R}^2)  = \tr(\chi_r^\infty T_{n, R}^2 \chi_r^\infty) = \tr\Big(\chi_r^\infty  \sgn( h_0^R) \sqrt{|  h_0^R|}   \Pi_n^{\mathfrak{q}}    h_0^R \Pi_n^{\mathfrak{q}}  \sqrt{|  h_0^R|}  \chi_r^\infty \Big) =  \tr(h_r^R   \Pi_n^{\mathfrak{q}}    h_0^R \Pi_n^{\mathfrak{q}}   ).
\end{align*}
Writing $h_0^R = h_0^r + h_r^R$ and applying equality \eqref{trace-square}, we get
\begin{align}\label{insert-com}
\begin{split}
& \tr(\chi_r^\infty T_{n, R}^2)  =  \tr(h_r^R   \Pi_n^{\mathfrak{q}}    h_r^R \Pi_n^{\mathfrak{q}}   ) + \tr(h_r^R   \Pi_n^{\mathfrak{q}}    h_0^r \Pi_n^{\mathfrak{q}}   )= \tr((h_r^R   \Pi_n^{\mathfrak{q}})^2 ) + \tr(h_r^R   \Pi_n^{\mathfrak{q}}    h_0^r \Pi_n^{\mathfrak{q}}  )
\\
= & \int_{r \le  | z| \le R}  h  (z)^2 \Pi_n^{\mathfrak{q}} (z,z) d\lambda_\phi(z) -  \frac{1}{2}\| [ \Pi_n^{\mathfrak{q}},   h_r^R ] \|_{HS}^2 +  \tr(h_r^R   \Pi_n^{\mathfrak{q}}    h_0^r \Pi_n^{\mathfrak{q}}  )
\\
=& \int_{r \le  | z| \le R}  h  (z)^2 \Pi_n (z,z) d\lambda_\phi(z)  + \int_{r \le  | z| \le R}  h  (z)^2 (\Pi_n^{\mathfrak{q}} (z,z)   - \Pi_n (z,z)) d\lambda_\phi(z)
\\
 & -  \frac{1}{2}\| [ \Pi_n^{\mathfrak{q}},   h_r^R ] \|_{HS}^2 +  \tr(h_r^R   \Pi_n^{\mathfrak{q}}    h_0^r \Pi_n^{\mathfrak{q}}  ).
\end{split}
\end{align}
Recall that since $\phi$ is radial we have
\begin{align}\label{kappa-vanish}
\int_{ r \le | z | \le R}   (\kappa(\mathfrak{p}, z) - \kappa(\mathfrak{q}, z)  ) \Pi_n(z,z)   d\lambda_\phi(z) =0.
\end{align}
The equalities \eqref{trace-away}, \eqref{insert-com} and  \eqref{kappa-vanish} together yield
\begin{align}\label{main-term}
\begin{split}
 \tr(\chi_r^\infty T_{n, R} ) - \frac{1}{2} \tr(\chi_r^\infty T_{n, R}^2)=&  \int_{r \le | z| \le R } \Big(h(z) -\frac{h(z)^2}{2}    + (\kappa(\mathfrak{p}, z) - \kappa(\mathfrak{q}, z) \Big)  \Pi_n(z,z) d\lambda_\phi(z)
 \\
 &+\int_{r \le | z| \le R }\Big(  h(z) - \frac{h(z)^2}{2} \Big)  (\Pi_n^{\mathfrak{q}}(z,z)-  \Pi_n(z,z) )  d\lambda_\phi(z)
\\
&  +\frac{1}{4}\| [ \Pi_n^{\mathfrak{q}},   h_r^R ] \|_{HS}^2  - \frac{1}{2}  \tr(h_r^R   \Pi_n^{\mathfrak{q}}    h_0^r \Pi_n^{\mathfrak{q}}  ).
\end{split}
\end{align}
Substituting \eqref{main-term} into \eqref{seperate-r}, we obtain the desired equality \eqref{4-terms}.

Proposition \ref{prop-factor} is proved completely.
\end{proof}

 Recalling notation  \eqref{notation-T}, for the regular factor $\Det_3(1  + T_{n, R})$ we have 
 \begin{lemma}\label{lem-S3}
The operator $T$ is in $\mathscr{C}_3(\mathscr{H})$.
\end{lemma}

\begin{proposition}\label{prop-order-det}
We have
\begin{align*}
&\lim_{R\to\infty} \Det_3(1+T_{n, R})  =  \Det_3(1+T_{n}); \,   \lim_{R\to\infty} \Det_3(1+T_{R})  =  \Det_3(1+T);
\\
&\lim_{n\to\infty} \Det_3(1+T_{n, R})  =  \Det_3(1+T_{R}); \,  \lim_{n\to\infty} \Det_3(1+T_{n})  =  \Det_3(1+T).
\end{align*}
In particular, we have
\[
\lim\limits_{n\to\infty}\lim\limits_{R\to\infty} \Det_3(1+T_{n, R})  = \lim\limits_{R\to\infty} \lim\limits_{n\to\infty}  \Det_3(1+T_{n, R}).
\]
\end{proposition}

Lemma \ref{lem-S3} and Proposition \ref{prop-order-det} will be proved in \S \ref{sub-sec-regular}.

 Recall the definition $I(n, R)$ in \eqref{int-sum-dec} and the decomposition of $I(n, R)$:
  \[
  I(n, R) =  E_1 (n, R; r) +  E_2 (n, R; r)  + E_3 (n, R; r) +  E_4 (n, R; r),
  \]
  where $E_1 (n, R; r), E_2 (n, R; r), E_3 (n, R; r),  E_4 (n, R; r)$ are  given in Proposition \ref{prop-factor}. Recall also the choice of $r_{\mathfrak{p}, \mathfrak{q}}>0$ in \eqref{close-1}.
The regularization factor $\exp(\tr(T_{n,R}) - \frac{1}{2} \tr(T_{n, R}^2))$ is controlled as follows.

  \begin{proposition}\label{prop-I-n-R}
Both the limits $\lim\limits_{R\to\infty} \lim\limits_{n\to\infty} I(n, R)$ and $\lim\limits_{n\to\infty} \lim\limits_{R\to\infty}  I (n, R)$ exist and we have
\begin{align}\label{}
\lim\limits_{R\to\infty} \lim\limits_{n\to\infty} I(n, R) =  \lim\limits_{n\to\infty} \lim\limits_{R\to\infty}  I (n, R).
\end{align}
  \end{proposition}

Proposition \ref{prop-I-n-R} follows from Lemmata \ref{lem-E1}, \ref{lem-E2} , \ref{lem-E3} , \ref{lem-E4}, formulated below and  proved in \S \ref{sub-sec-4E}.

  \begin{lemma}\label{lem-E1}
  For any $r>r_{\mathfrak{p}, \mathfrak{q}}$, we have
\begin{align*}
\lim_{R\to\infty} \lim_{n\to\infty} E_1 (n, R; r) =  \lim_{n\to\infty} \lim_{R\to\infty} E_1 (n, R; r) &=    \tr(\chi_0^r T \chi_0^r ) - \frac{1}{2} \tr(\chi_0^r T^2 \chi_0^r)     - \frac{1}{2}  \tr(h_r^\infty   \Pi^{\mathfrak{q}}    h_0^r \Pi^{\mathfrak{q}}  )  .
\end{align*}
  \end{lemma}

 \begin{lemma}\label{lem-E2}
  For any $r>r_{\mathfrak{p}, \mathfrak{q}}$, we have
\begin{align}\label{2-int}
\lim_{R\to\infty} \lim_{n\to\infty} E_2 (n, R; r) =  \lim_{n\to\infty} \lim_{R\to\infty} E_2 (n, R; r) &= \int_{ | z | \ge r}    \Big( h(z) - \frac{h(z)^2}{2}   + (\kappa(\mathfrak{p}, z) - \kappa(\mathfrak{q}, z)) \Big) \Pi(z,z)   d\lambda_\phi(z) .
\end{align}
\end{lemma}

   \begin{lemma}\label{lem-E3}
For any $\varepsilon>0$, there exists $r_\varepsilon>r_{\mathfrak{p}, \mathfrak{q}}$ such that  if $r\ge r_\varepsilon$, then
\begin{align}\label{unif-es}
\sup_{n\ge \ell, \, R >r}  | E_3(n, R; r)| \le \varepsilon.
\end{align}
  \end{lemma}

  \begin{lemma}\label{lem-E4}
  For any $\varepsilon>0$, there exists $r_\varepsilon>r_{\mathfrak{p}, \mathfrak{q}}$ such that  if $r\ge r_\varepsilon$, then
\begin{align}\label{unif-es-4}
\sup_{n\ge \ell, \, R >r}  | E_4(n, R; r)| \le \varepsilon.
\end{align}
  \end{lemma}

We now derive Proposition \ref{prop-I-n-R} from Lemmata \ref{lem-E1}, \ref{lem-E2}, \ref{lem-E3}, \ref{lem-E4}.

  \begin{proof}[Proof of Proposition \ref{prop-I-n-R}]
  First,  by Propositions \ref{prop-finite}, \ref{prop-factor}, \ref{prop-order-det}, the limit $ \lim\limits_{n\to\infty} \lim\limits_{R\to\infty} I(n, R)$ exists and by the equality \eqref{R-n} and by Propositions  \ref{prop-factor}, \ref{prop-order-det},  the limit $ \lim\limits_{R\to\infty}  \lim\limits_{n\to\infty} I(n, R)$ exists.

   Now by Lemmata \ref{lem-E1}, \ref{lem-E2}, \ref{lem-E3} and \ref{lem-E4},  for any $\varepsilon>0$, there exists $r_{\varepsilon}>0$, such that we may write
\[
I(n, R) = I_1(n, R; \varepsilon) + I_2(n, R; \varepsilon)
\]
 in such a way that
\begin{align}
 \lim_{n\to\infty} \lim_{R\to\infty} I_1(n, R; \varepsilon) =   \lim_{R\to\infty}  \lim_{n\to\infty} I_1(n, R; \varepsilon)  \an | I_2(n, R; \varepsilon) | \le \varepsilon  \text{ for any $n\ge \ell$ and $R >r_\varepsilon$}.
\end{align}
For any $\varepsilon>0$, let us denote
\[
I_1(\varepsilon) : =\lim_{n\to\infty} \lim_{R\to\infty} I_1(n, R; \varepsilon) =   \lim_{R\to\infty}  \lim_{n\to\infty} F_1(n, R; \varepsilon).
\]
Then we have
\[
| \lim_{n\to\infty} \lim_{R\to\infty} I(n, R) - I_1(\varepsilon) | \le \varepsilon \an
| \lim_{R\to\infty} \lim_{n\to\infty}  I(n, R) - I_1(\varepsilon) | \le \varepsilon.
\]
Consequently, we obtain that
\[
 \lim_{R\to\infty} \lim_{n\to\infty}  I(n, R) = \lim_{n\to\infty}  \lim_{R\to\infty}   I(n, R).
\]
Proposition \ref{prop-I-n-R} is proved  completely.
  \end{proof}

 \begin{proof}[Proof of Proposition \ref{prop-order}]
By \eqref{exp-decom}, we have
\begin{align*}
 \E_{\Pi_n^{\mathfrak{q}}}  \Big[\prod_{x \in \mathscr{X}: | x| \le R} \prod_{i= 1}^\ell \Big| \frac{x-p_i}{x-q_i}\Big|^2\Big]  =  \Det_3(1 + T_{n, R}) \cdot \exp[I(n, R)].
\end{align*}
Now by Proposition \ref{prop-order-det} and Proposition \ref{prop-I-n-R}, we may exchange the two limits as $n\to\infty$ and as $R\to\infty$ and get the desired equality
\begin{align*}
 \lim_{R \to\infty}  \lim_{n\to\infty}  \E_{\Pi_n^{\mathfrak{q}}}  \Big[\prod_{x \in \mathscr{X}: | x| \le R} \prod_{i= 1}^\ell \Big| \frac{x-p_i}{x-q_i}\Big|^2\Big]  = \lim_{n\to\infty}\lim_{R \to\infty}  \E_{\Pi_n^{\mathfrak{q}}}  \Big[\prod_{x \in \mathscr{X}: | x| \le R} \prod_{i= 1}^\ell \Big| \frac{x-p_i}{x-q_i}\Big|^2\Big].
\end{align*}
\end{proof}

 \subsubsection{Control of the regular part}\label{sub-sec-regular}

In this section, we prove Lemma \ref{lem-S3} and Proposition \ref{prop-order-det}.

 \begin{lemma}[{\cite[Lemma 5.3]{QB3}}]\label{lem-com}
For any $r \ge r_{\mathfrak{p}, \mathfrak{q}}$, we have
\begin{align}\label{com-1}
\|[h,    \chi_r^\infty \Pi \chi_{r}^\infty ]\|_{HS}^2 = \iint_{\C^2}  | h(z) - h(w) |^2  | \Pi(z,w)|^2   \chi_r^\infty(z) \chi_r^\infty(w) d\lambda_\phi(z) d\lambda_\phi(w)  < \infty;
\end{align}
\begin{align}\label{com-3}
\|[h_r^\infty,    \Pi ]\|_{HS}^2 = \iint_{\C^2}  | h_r^\infty(z) - h_r^\infty(w) |^2 | \Pi(z,w)|^2   d\lambda_\phi(z) d\lambda_\phi(w) < \infty.
\end{align}
\end{lemma}

\begin{proof}
The inequality \eqref{com-1} is proved in \cite[Lemma 5.3]{QB3}. Since
\begin{align}\label{3-int}
\begin{split}
 \|[h_r^\infty,    \Pi ]\|_{HS}^2 =& \int_{|z| \le r} \int_{w \in \C} |h_r^\infty(w) |^2 | \Pi(z,w)|^2   d\lambda_\phi(z) d\lambda_\phi(w)   +  \int_{z \in  \C} \int_{|w|\le r} | h_r^\infty(z) |^2 | \Pi(z,w)|^2   d\lambda_\phi(z) d\lambda_\phi(w)
 \\
 & +  \int_{|z| \ge r} \int_{|w|\ge r} | h(z) - h(w) |^2 | \Pi(z,w)|^2   d\lambda_\phi(z) d\lambda_\phi(w).
 \end{split}
\end{align}
The first and the second integrals in \eqref{3-int} are equal and are majorated by
\[
\| h_r^\infty  \|_\infty^2   \int_{|z| \le r} \int_{w \in \C}    |\Pi(z,w)|^2  d\lambda_\phi(z) d\lambda_\phi(w)  \le   \| h_r^\infty  \|_\infty^2   \int_{|z| \le r}   \Pi(z,z)  d\lambda_\phi(z) < \infty.
\]
The third integral in \eqref{3-int} is finite by \eqref{com-1}.
\end{proof}

\begin{lemma}\label{lem-com-q}
For any $r \ge r_{\mathfrak{p}, \mathfrak{q}}$, we have
\begin{align}\label{com-2}
\|[h,    \chi_r^\infty \Pi^{\mathfrak{q}} \chi_r^\infty ]\|_{HS}^2 = \iint_{\C^2}  | h(z) - h(w) |^2 | \Pi^{\mathfrak{q}}(z,w)|^2  \chi_r^\infty(z) \chi_r^\infty(w) d\lambda_\phi(z) d\lambda_\phi(w)  < \infty;
\end{align}
\begin{align}\label{com-4}
\|[h_r^\infty,    \Pi^{\mathfrak{q}} ]\|_{HS}^2 = \iint_{\C^2}  | h_r^\infty(z) - h_r^\infty(w) |^2 | \Pi^{\mathfrak{q}}(z,w)|^2   d\lambda_\phi(z) d\lambda_\phi(w) < \infty.
\end{align}
\end{lemma}
\begin{proof}
Since  $\Pi^{\mathfrak{q}}$ is a finite rank perturbation of $\Pi$, the inequalities \eqref{com-2} and \eqref{com-4} follow from the inequality \eqref{com-1} and \eqref{com-3} respectively.
\end{proof}

\begin{proof}[Proof of Lemma \ref{lem-S3}]
It suffices to show that $ | h|^{1/2} \cdot \Pi^{\mathfrak{q}}  \cdot |h|^{1/2} \in \mathscr{C}_3(\mathscr{H})$.  We have
\begin{align*}
\Big|  | h|^{1/2} \cdot \Pi^{\mathfrak{q}}  \cdot |h|^{1/2}\Big|^3 =  | h|^{1/2} \cdot \Pi^{\mathfrak{q}}  \cdot |h| \cdot \Pi^{\mathfrak{q}}  \cdot | h| \cdot \Pi^{\mathfrak{q}}  \cdot |h|^{1/2}.
\end{align*}
It suffices to show that $ |h|^{1/2}  \Pi^{\mathfrak{q}}   |h| \Pi^{\mathfrak{q}} $ is Hilbert-Schmidt. To this end, write
\[
|h|^{1/2} \Pi^{\mathfrak{q}} |h|\Pi^{\mathfrak{q}}=  \chi_r^\infty |h|^{1/2} \Pi^{\mathfrak{q}} \chi_r^\infty    |h|\Pi^{\mathfrak{q}} +\chi_r^\infty |h|^{1/2} \Pi^{\mathfrak{q}} \chi_0^r |h|\Pi^{\mathfrak{q}} +\chi_0^r |h|^{1/2} \Pi^{\mathfrak{q}} \chi_r^\infty    |h|\Pi^{\mathfrak{q}} +\chi_0^r |h|^{1/2} \Pi^{\mathfrak{q}} \chi_0^r |h|\Pi^{\mathfrak{q}}.
\]
Observe that
\begin{align}\label{sum-HS}
 \chi_r^\infty |h|^{1/2}    \Pi^{\mathfrak{q}} \chi_r^\infty    |h|  \Pi^{\mathfrak{q}} -   \chi_r^\infty |h|^{1/2} h  \Pi^{\mathfrak{q}} \chi_r^\infty   \sgn(h) \Pi^{\mathfrak{q}} =  \chi_r^\infty |h|^{1/2} [ \chi_r^\infty  \Pi^{\mathfrak{q}} \chi_r^\infty,  \,\,   h]\sgn(h)  \Pi^{\mathfrak{q}}.
\end{align}
Since $\chi_r^\infty | h|^{1/2}$ is bounded, we may apply Lemma \ref{lem-com-q} to conclude that
\[
 \chi_r^\infty |h|^{1/2} [ \chi_r^\infty  \Pi^{\mathfrak{q}} \chi_r^\infty,   \,  h]\sgn(h)  \Pi^{\mathfrak{q}} \in \mathscr{C}_2(\mathscr{H}).
 \]
 Note also that $h(z) = \mathcal{O}(1/|z|)$ as $|z|\to\infty$,  hence  by Lemma \ref{lem-o3}, we have
 \[
 \|\chi_r^\infty |h|^{1/2} h \Pi^{\mathfrak{q}}\|_{HS}^2 = \int_{|z|\ge r} |h(z)|^3 \Pi^{\mathfrak{q}}(z,z) d\lambda_\phi(z) \lesssim   \int_{|z|\ge r} \frac{1}{|z|^3} \Pi^{\mathfrak{q}}(z,z) d\lambda_\phi(z)  < \infty.
 \]
 It follows that $\chi_r^\infty |h|^{1/2} h  \Pi^{\mathfrak{q}} \chi_r^\infty   \sgn(h) \Pi^{\mathfrak{q}} $ is Hilbert-Schmidt. Consequently, by  \eqref{sum-HS},  the operator $ \chi_r^\infty |h|^{1/2}    \Pi^{\mathfrak{q}} \chi_r^\infty    |h|  \Pi^{\mathfrak{q}}$, a sum of two Hilbert-Schmidt operators,  is itself  Hilbert-Schmidt.

Now we show that $\Pi^{\mathfrak{q}} \chi_0^r |h|\Pi^{\mathfrak{q}}$ is Hilbert-Schmidt. Indeed,  since  there exists a constant $C> 0$ such that
\begin{align}\label{near-q}
|\Pi^{\mathfrak{q}}(z,w)|\le C \cdot \prod_{i=1}^\ell | (z-q_i)(w-q_i)|  \text{ for any $|z|< r$ and $|w|\le r$},
\end{align}
we have
\[
\|\Pi^{\mathfrak{q}} \chi_0^r |h|\Pi^{\mathfrak{q}}\|_{HS}^2 = \| \chi_0^r  |h|^{1/2}\Pi^{\mathfrak{q}} \chi_0^r  |h|^{1/2} \|_{HS}^2 = \int_{|z|\le r}\int_{|w|\le r} |h(z)h(w)| \cdot | \Pi^{\mathfrak{q}}(z,w)|^2 d\lambda_\phi(z) d\lambda_\phi(w)<\infty.
\]
Consequently, $\chi_r^\infty |h|^{1/2} \Pi^{\mathfrak{q}} \chi_0^r |h|\Pi^{\mathfrak{q}} $ is Hilbert-Schmidt.

We show also that $\chi_0^r |h|^{1/2} \Pi^{\mathfrak{q}}$ is Hilbert-Schmidt. Indeed,
\[
\|\chi_0^r |h|^{1/2} \Pi^{\mathfrak{q}}\|_{HS}^2 = \int_{|z|\le r} \int_{w\in \C}|h(z)| \cdot |\Pi^{\mathfrak{q}}(z,w)|^2d\lambda_\phi(z) d\lambda_\phi(w) = \int_{|z|\le r} |h(z)| \Pi^{\mathfrak{q}}(z,z)d\lambda_\phi(z) <\infty,
\]
where we used again \eqref{near-q} for $z = w$ and $|z| \le r$.  Now since  $\chi_r^\infty |h|\Pi^{\mathfrak{q}}$ and $\chi_0^r |h|\Pi^{\mathfrak{q}}$ are both bounded operator, we conclude that $\chi_0^r |h|^{1/2} \Pi^{\mathfrak{q}} \chi_r^\infty    |h|\Pi^{\mathfrak{q}}$ and $\chi_0^r |h|^{1/2} \Pi^{\mathfrak{q}} \chi_0^r |h|\Pi^{\mathfrak{q}}$ are both Hilbert-Schmidt.

Lemma \ref{lem-S3} is proved completely.
\end{proof}

\begin{proof}[Proof of Proposition \ref{prop-order-det}]
By  \eqref{cont-det}, it suffices to prove the corresponding convergences of operators in $\mathscr{C}_3(\mathscr{H})$.   By Lemma \ref{lem-S3} and  Proposition \ref{prop-cp}, we have
\begin{align}\label{C3-cv}
T_{n, R}\xrightarrow[R\to \infty]{\text{in $\mathscr{C}_3(\mathscr{K})$}} T_n;  \quad T_{R}\xrightarrow[R\to \infty]{\text{in $\mathscr{C}_3(\mathscr{K})$}} T.
\end{align}
By Lemma \ref{lem-S3}, we also have $ \sqrt{|h|} \Pi^{\mathfrak{q}} \in \mathscr{C}_6(\mathscr{H})$. Applying Proposition \ref{prop-cp} again and  noting that $\sqrt{|h|} \Pi_n^{\mathfrak{q}} = \sqrt{|h|} \Pi^{\mathfrak{q}} \cdot \Pi_n^{\mathfrak{q}}$, we obtain
\[
\sqrt{|h|} \Pi_n^{\mathfrak{q}} \xrightarrow[n \to \infty]{\text{in $\mathscr{C}_6(\mathscr{K})$}} \sqrt{|h|} \Pi^{\mathfrak{q}}.
\]
The above convergence, combined with the H\"older inequalities  \eqref{S-holder} for operators in von~Neumann-Schatten classes  immediately yields the desired convergences in \eqref{C3-cv}. Proposition \ref{prop-order-det} is proved completely.
\end{proof}

\subsubsection{Control of the regularization factor}\label{sub-sec-4E}

Recall the notation introduced in \eqref{notation-trunc} and \eqref{notation-T}.

 \bigskip

\noindent \textbf {1. Control of $E_1(n, R; r)$.}

\begin{proof}[Proof of Lemma \ref{lem-E1}]
Note that
\[
E_1(n, R; r) =    \tr(\chi_0^r T_{n, R} \chi_0^r ) - \frac{1}{2} \tr(\chi_0^r T_{n, R}^2 \chi_0^r)     - \frac{1}{2}  \tr(h_r^R   \Pi_n^{\mathfrak{q}}    h_0^r \Pi_n^{\mathfrak{q}}  ).
\]
For proving Lemma \ref{lem-E1},  it suffices to prove that for any $r\ge r_{\mathfrak{p}, \mathfrak{q}}$, we have the following convergences in $\mathscr{C}_1(\mathscr{H})$:
\begin{align}
\lim_{R\to\infty} \lim_{n\to\infty} \chi_0^r T_{n, R} \chi_0^r &=  \lim_{n\to\infty} \lim_{R\to\infty} \chi_0^r T_{n, R} \chi_0^r= \chi_0^r T \chi_0^r;\label{1-tr-cv}
\\
\lim_{R\to\infty} \lim_{n\to\infty} \chi_0^r T^2_{n, R} \chi_0^r &=  \lim_{n\to\infty} \lim_{R\to\infty} \chi_0^r T^2_{n, R} \chi_0^r= \chi_0^r T^2 \chi_0^r; \label{2-tr-cv}
\\
\lim_{R\to\infty} \lim_{n\to\infty}  h_r^R   \Pi_n^{\mathfrak{q}}    h_0^r \Pi_n^{\mathfrak{q}}   &=   \lim_{n\to\infty} \lim_{R\to\infty}  h_r^R   \Pi_n^{\mathfrak{q}}    h_0^r \Pi_n^{\mathfrak{q}} =  h_r^\infty   \Pi^{\mathfrak{q}}    h_0^r \Pi^{\mathfrak{q}}.\label{3-tr-cv}
\end{align}
Let us check the convergences in \eqref{2-tr-cv}. We may write
\[
 \chi_0^r T^2_{n, R} \chi_0^r  =  (\chi_0^r   \sgn(h) \sqrt{| h|}    \Pi^{\mathfrak{q}}  \cdot  \Pi_n^{\mathfrak{q}}  ) \cdot (  \Pi^{\mathfrak{q}}     \sqrt{|h|}       \chi_0^R   \sgn(h) \sqrt{| h|}   \Pi^{\mathfrak{q}}  )\cdot(\Pi_n^{\mathfrak{q}}  \cdot  \Pi^{\mathfrak{q}}  \sqrt{|h|}          \chi_0^r ).
\]
Since $ \chi_0^r \sgn(h) \sqrt{|h|} \Pi^{\mathfrak{q}}$ and  $ \Pi^{\mathfrak{q}} \sqrt{|h|}\chi_0^r$ are Hilbert-Schmidt, we may apply Proposition \ref{prop-cp} to conclude that
\[
 \chi_0^r \sgn(h) \sqrt{|h|} \Pi^{\mathfrak{q}} \cdot \Pi^{\mathfrak{q}}_n  \xrightarrow[n\to\infty]{\text{in $\mathscr{C}_2(\mathscr{H})$}}  \chi_0^r \sgn(h) \sqrt{|h|} \Pi^{\mathfrak{q}}; \quad
\Pi^{\mathfrak{q}}_n  \cdot \Pi^{\mathfrak{q}} \sqrt{|h|}\chi_0^r \xrightarrow[n\to\infty]{\text{in $\mathscr{C}_2(\mathscr{H})$}}  \Pi^{\mathfrak{q}} \sqrt{|h|}\chi_0^r.
\]
It follows, by using also  the fact that $\Pi^{\mathfrak{q}}     \sqrt{|h|}       \chi_0^R   \sgn(h) \sqrt{| h|}   \Pi^{\mathfrak{q}}$ is bounded,  that
\[
 \chi_0^r T^2_{n, R} \chi_0^r \xrightarrow[n\to\infty]{ \text{in $\mathscr{C}_1(\mathscr{H})$}} (\chi_0^r   \sgn(h) \sqrt{| h|}    \Pi^{\mathfrak{q}}   ) \cdot (  \Pi^{\mathfrak{q}}     \sqrt{|h|}       \chi_0^R   \sgn(h) \sqrt{| h|}   \Pi^{\mathfrak{q}}  )\cdot( \Pi^{\mathfrak{q}}  \sqrt{|h|}          \chi_0^r ) = \chi_0^r T^2_{R} \chi_0^r  .
 \]
 Now by writing
 \[
 \chi_0^r T^2_{R} \chi_0^r = \Big[(\chi_0^r   \sgn(h) \sqrt{| h|}    \Pi^{\mathfrak{q}}  \cdot \Pi^{\mathfrak{q}}     \sqrt{|h|} )      \chi_0^R\Big]   \cdot \Big[\chi_0^R  ( \sgn(h) \sqrt{| h|}   \Pi^{\mathfrak{q}}  \cdot \Pi^{\mathfrak{q}}  \sqrt{|h|}          \chi_0^r ) \Big],
 \]
also by using the fact that
\[
\chi_0^r   \sgn(h) \sqrt{| h|}    \Pi^{\mathfrak{q}}  \cdot \Pi^{\mathfrak{q}}     \sqrt{|h|} \an  \sgn(h) \sqrt{| h|}   \Pi^{\mathfrak{q}}  \cdot \Pi^{\mathfrak{q}}  \sqrt{|h|}          \chi_0^r
\]
 are both Hilbert-Schmidt, we may apply Proposition \ref{prop-cp}  to conclude that
\[
 \chi_0^r T^2_{R} \chi_0^r \xrightarrow[R\to\infty]{\text{in $\mathscr{C}_1(\mathscr{H})$}}  (\chi_0^r   \sgn(h) \sqrt{| h|}    \Pi^{\mathfrak{q}}  \cdot \Pi^{\mathfrak{q}}     \sqrt{|h|} )    \cdot ( \sgn(h) \sqrt{| h|}   \Pi^{\mathfrak{q}}  \cdot \Pi^{\mathfrak{q}}  \sqrt{|h|}          \chi_0^r ) =  \chi_0^r T^2 \chi_0^r.
\]
Now  we obtain the following convergence in $\mathscr{C}_1(\mathscr{H})$:
\[
\lim_{R\to\infty} \lim_{n\to\infty} \chi_0^r T^2_{n, R} \chi_0^r= \chi_0^r T^2 \chi_0^r.
\]
In a similar way, we obtain also the following convergence in $\mathscr{C}_1(\mathscr{H})$:
\[
\lim_{n\to\infty} \lim_{R\to\infty}  \chi_0^r T^2_{n, R} \chi_0^r= \chi_0^r T^2 \chi_0^r.
\]
The convergences in \eqref{2-tr-cv} is proved  completely.

This argument also yields the  convergences in \eqref{1-tr-cv} and \eqref{3-tr-cv}.
\end{proof}

\bigskip

\noindent \textbf {2. Control of $E_2(n, R; r)$.}

\begin{lemma}\label{lem-h-k}
For any $r\ge r_{\mathfrak{p}, \mathfrak{q}}$, the following integrals are finite:
\[
\int_{| z |\ge r}   \Big| h (z) - \frac{h(z)^2}{2}  + ( \kappa(\mathfrak{p}, z) - \kappa(\mathfrak{q},z))  \Big| \Pi(z,z) d\lambda_\phi(z) < \infty;
\]
\[
\int_{| z |\ge r}   \Big| h (z) - \frac{h(z)^2}{2} + ( \kappa(\mathfrak{p}, z) - \kappa(\mathfrak{q}, z))  \Big| \Pi^{\mathfrak{q}}(z,z) d\lambda_\phi(z) < \infty.
\]
\end{lemma}

\begin{proof}
The second inequality follows immediately from the first one.  By Lemma \ref{lem-o3}, it suffices to prove that
\[
h (z) - \frac{h(z)^2}{2}  =    (\kappa(\mathfrak{q}, z) - \kappa(\mathfrak{p}, z) ) + \mathcal{O} (1/|z|^3) \text{ as } |z| \to\infty.
\]
To this end, we first write $\alpha_i (z) = \frac{ q_i- p_i}{z-q_i}$ and $\beta_i(z) = \alpha_i(z) + \overline{\alpha_i(z)}$.  Then we have
\begin{align*}
h (z)   & = \prod_{i = 1}^\ell | 1+ \alpha_i (z) |^2 -1 =  \prod_{ i= 1}^\ell (1 + \beta_i(z) + | \alpha_i(z)|^2) -1
\\
& =  \sum_{ i = 1}^\ell  ( \beta_i (z) + |\alpha_i(z)|^2) +  \sum_{1\le i < j \le \ell} \beta_i(z) \beta_j(z)   + \mathcal{O} (1/|z|^3) \text{ as } |z| \to\infty.
\end{align*}
It follows that
\begin{align*}
h(z)^2&=  \sum_{i = 1}^\ell \beta_i(z)^2  + 2 \sum_{1\le i< j \le \ell}  \beta_i(z) \beta_j(z) +  \mathcal{O} (1/|z|^3)
\\
&=   \sum_{i=1}^\ell  (\alpha_i(z)^2 + \overline{\alpha_i(z)}^2 + 2 | \alpha_i(z)|^2)   +  2 \sum_{1\le i< j \le \ell}  \beta_i(z) \beta_j(z) +    \mathcal{O} (1/|z|^3) \text{ as } |z| \to\infty.
\end{align*}
Consequently, we have
\begin{align*}
h(z) - \frac{h(z)^2}{2} = \sum_{i=1}^\ell   \Big( \beta_i(z) - \frac{1}{2} \alpha_i(z)^2 -  \frac{1}{2} \overline{\alpha_i(z)}^2\Big) +    \mathcal{O} (1/|z|^3) \text{ as } |z| \to\infty.
\end{align*}
Equality \eqref{cocyle-sim}  implies
\begin{align*}
 \beta_i(z) - \frac{1}{2} \alpha_i(z)^2 -  \frac{1}{2} \overline{\alpha_i(z)}^2   = \kappa(q_i, z)  - \kappa(p_i, z) +  \mathcal{O}(1/|z|^3) \text{ as } |z| \to\infty.
\end{align*}
Combining the two equations, we complete the proof of  Lemma \ref{lem-h-k}.
\end{proof}

\begin{proof}[Proof of Lemma \ref{lem-E2}]
Recall that
\[
E_2(n, R; r) = \int_{ r \le | z | \le R}    \Big( h(z) - \frac{h(z)^2}{2}   + (\kappa(\mathfrak{p}, z) - \kappa(\mathfrak{q}, z)) \Big) \Pi_n(z,z)   d\lambda_\phi(z)
\]
Let $n\ge \ell$.
On the one hand,  since the function $ \Big( h(z) - \frac{h(z)^2}{2}   + (\kappa(\mathfrak{p}, z) - \kappa(\mathfrak{q}, z)) \Big) \Pi_n(z,z)$ is integrable on $\{z\in\C: |z|\ge r\}$,  we have
\[
  \lim_{R\to\infty} E_2 (n, R; r)  = \int_{| z | \ge r}    \Big( h(z) - \frac{h(z)^2}{2}   + (\kappa(\mathfrak{p}, z) - \kappa(\mathfrak{q}, z)) \Big) \Pi_n(z,z)   d\lambda_\phi(z).
\]
Taking into account Lemma \ref{lem-h-k},  using the clear inequality $\Pi_n(z,z) \le \Pi(z,z)$ and the Dominated Convergence Theorem, we obtain
\[
\lim_{n\to\infty} \lim_{R\to\infty} E_2 (n, R; r)  = \int_{| z | \ge r}    \Big( h(z) - \frac{h(z)^2}{2}   +(\kappa(\mathfrak{p}, z) - \kappa(\mathfrak{q}, z)) \Big) \Pi(z,z)   d\lambda_\phi(z).
\]
On the other hand,  by the Dominated Convergence Theorem, we also have
\[
\lim_{n\to\infty} E_2 (n, R; r)  = \int_{r \le | z | \le R}    \Big( h(z) - \frac{h(z)^2}{2}   +(\kappa(\mathfrak{p}, z) - \kappa(\mathfrak{q}, z)) \Big) \Pi(z,z)   d\lambda_\phi(z).
\]
Hence
\[
\lim_{R\to\infty}\lim_{n\to\infty} E_2 (n, R; r)   = \int_{| z | \ge r}    \Big( h(z) - \frac{h(z)^2}{2}   +(\kappa(\mathfrak{p}, z) - \kappa(\mathfrak{q}, z)) \Big) \Pi(z,z)   d\lambda_\phi(z).
\]
Equality \eqref{2-int} is proved completely.
\end{proof}

\bigskip

\noindent \textbf {3. Control of $E_3(n, R; r)$. }
\begin{proof}[Proof of Lemma \ref{lem-E3}]
For any fixed $n\ge\ell$ and any pair of positive numbers $r, R$ satisfying $R>r \ge r_{\mathfrak{p}, \mathfrak{q}}$,  we have
\begin{align*}
&|E_3(n, R; r) |= \int_{ r \le | z | \le R}    \Big|  h(z) - \frac{h(z)^2}{2} \Big| (\Pi_n(z,z)    - \Pi_n^{\mathfrak{q}}(z,z)) d\lambda_\phi(z)
\\
& \le \sup_{|z|\ge r} \Big|  h(z) - \frac{h(z)^2}{2} \Big|  \cdot  \int_{ \C}   (\Pi_n(z,z)    - \Pi_n^{\mathfrak{q}}(z,z)) d\lambda_\phi(z) = \ell \cdot  \sup_{|z|\ge r} \Big|  h(z) - \frac{h(z)^2}{2} \Big|.
\end{align*}
It follows that
\begin{align*}
\limsup_{r\to\infty}\sup_{n\ge\ell, \, R> r}  | E_3(n, R; r)|  \le \lim_{r\to\infty}  \ell \cdot  \sup_{|z|\ge r} \Big|  h(z) - \frac{h(z)^2}{2} \Big| =0.
\end{align*}
Lemma \ref{lem-E3} is proved completely.
\end{proof}

\bigskip

\noindent \textbf {4. Control of $E_4(n, R; r)$. }

 \begin{lemma}\label{lem-h-h}
There exits a constant $C>0$, such that
 \begin{align}\label{d-h-h}
| h(z) - h(w)|  \le C \Big|  \frac{1}{z}- \frac{1}{w}\Big|, \text{ if $|z| \ge r_{\mathfrak{p}, \mathfrak{q}}$ and $| w| \ge r_{\mathfrak{p}, \mathfrak{q}}$}.
 \end{align}
 \end{lemma}

 \begin{proof}
Clearly, there exist $\gamma_1, \cdots, \gamma_\ell  \in\C$, such that
\begin{align*}
g(z) = \left|   1  + \sum_{k=1}^\ell \frac{\gamma_k}{z-q_k} \right|^2; \quad h(z) = \left|   1  + \sum_{k=1}^\ell \frac{\gamma_k}{z-q_k} \right|^2 -1.
\end{align*}
Consequently, if $|z|\ge r_{\mathfrak{p}, \mathfrak{q}}$ and $|w|\ge r_{\mathfrak{p}, \mathfrak{q}}$, then
\begin{align*}
  | h (z)-  h (w)|    \le \sup_{|z|\ge r_{\mathfrak{p}, \mathfrak{q}}, |w| \ge r_{\mathfrak{p}, \mathfrak{q}}}   \left(\Big|   1  + \sum_{k=1}^\ell \frac{\gamma_k}{z-q_k} \Big| + \Big|   1  + \sum_{k=1}^\ell \frac{\gamma_k}{w-q_k} \Big|\right)  \cdot \left|   \sum_{k=1}^\ell \frac{\gamma_k}{z-q_k}  -   \frac{\gamma_k}{w-q_k}   \right|.
\end{align*}
The  simple inequality
\begin{align*}
 \displaystyle \sup_{|z|\ge r_{\mathfrak{p}, \mathfrak{q}}, |w|\ge r_{\mathfrak{p}, \mathfrak{q}} }  \displaystyle \frac{ \left|   \frac{1}{z-q_k}  -   \frac{1}{w-q_k}   \right|}{ \left|  \frac{1}{z}  -   \frac{1}{w}   \right|} < \infty,
\end{align*}
implies now the existence of $C>0$ such that  \eqref{d-h-h} holds.
 \end{proof}

\begin{lemma}\label{lem-real}
   For any $r \ge r_{\mathfrak{p}, \mathfrak{q}}$, we have
\begin{align}\label{simple-com}
\int_{|z|\ge r } \int_{|w|\ge r }\Big|\frac{1}{z} - \frac{1}{w}\Big|^2 \cdot | \Pi(z,w)|^2 d\lambda_\phi(z)\lambda_\phi(w) <\infty.
\end{align}
Moreover, the following limit holds:
\begin{align}\label{off-diag}
\lim_{r\to\infty} \int_{|z|\le r}\int_{|w|\ge r } \frac{1}{|w|^2}  \cdot | \Pi(z,w)|^2 d\lambda_\phi(z)\lambda_\phi(w)=0.
\end{align}
\end{lemma}

 \begin{proof}
 The explicit computations in  \cite[Lemma 5.3, 5.4]{QB3} indeed give the desired relations \eqref{simple-com} and \eqref{off-diag}.
 \end{proof}

\begin{proof}[Proof of Lemma \ref{lem-E4}]
By writing $h_r^R = h_r^\infty - h_R^\infty$, we have
\[
[h_r^R, \Pi_n^{\mathfrak{q}}] =  [h_r^\infty, \Pi_n^{\mathfrak{q}}] - [h_R^\infty, \Pi_n^{\mathfrak{q}}].
\]
Consequently, for proving Lemma \ref{lem-E4},  it suffices to prove that
\[
\lim_{r\to\infty}  \sup_{n\ge \ell}\|  [ h_r^\infty,  \Pi_n^{\mathfrak{q}}] \|_{HS} =0.
\]
However, since
\begin{align*}
&\sup_{n\ge \ell}\Big| \|  [ h_r^\infty,  \Pi_n^{\mathfrak{q}}] \|_{HS} - \|  [ h_r^\infty,  \Pi_n] \|_{HS}\Big|  \le \sup_{n\ge \ell}  \|   [ h_r^\infty,  \Pi_n^{\mathfrak{q}} - \Pi_n]   \|_{HS}
\\
& \le 2 \|h_r^\infty\|_\infty \cdot \sup_{n\ge \ell} \| \Pi_n^{\mathfrak{q}} - \Pi_n \|_{HS} \le  2 \|h_r^\infty\|_\infty  \sqrt{\ell} \xrightarrow{r\to\infty} 0,
\end{align*}
it suffices to show that
\begin{align}\label{small-com}
\lim_{r\to\infty}  \sup_{n\ge \ell}\|  [ h_r^\infty,  \Pi_n] \|_{HS} =0.
\end{align}
To this end, by noting $|\Pi^{\mathfrak{q}}(z,w)| = |\Pi^{\mathfrak{q}}(w, z)|$, we have the following identity:
\begin{align*}
 \|  [ h_r^\infty,  \Pi_n] \|_{HS}^2 =& 2  \int_{|z|\le r}\int_{|w|\ge r} |h(w)|^2 | \Pi_n(z,w)|^2 d\lambda_\phi(z)\lambda_\phi(w)
 \\
 & +  \int_{|z|\ge r } \int_{|w|\ge r } | h(z)-h(w)|^2 | \Pi_n(z,w)|^2 d\lambda_\phi(z)\lambda_\phi(w).
\end{align*}
It follows from Lemma \ref{lem-h-h} and the elementary estimate $|h(z)|  = \mathcal{O}(1/|z|)$ as $|z|\to\infty$,  that there exists  $C>0$ such  that
\begin{align}\label{control-I-I}
 \|  [ h_r^\infty,  \Pi_n] \|_{HS}^2 \le & C(I_1(n, r) + I_2(n, r)),
 \end{align}
 where
 \begin{align*}
  I_1(n,r)& :=  \int_{|z|\le r}\int_{|w|\ge r} \frac{1}{|w|^2}  \cdot | \Pi_n(z,w)|^2 d\lambda_\phi(z)\lambda_\phi(w);
 \\
 I_2(n,r) & := \int_{|z|\ge r } \int_{|w|\ge r }\Big|\frac{1}{z} - \frac{1}{w}\Big|^2 | \Pi_n(z,w)|^2 d\lambda_\phi(z)\lambda_\phi(w).
\end{align*}
Similarly, let us denote
 \begin{align*}
  I_1(r)& :=  \int_{|z|\le r}\int_{|w|\ge r} \frac{1}{|w|^2}  \cdot | \Pi(z,w)|^2 d\lambda_\phi(z)\lambda_\phi(w);
 \\
 I_2(r) & := \int_{|z|\ge r } \int_{|w|\ge r }\Big|\frac{1}{z} - \frac{1}{w}\Big|^2 | \Pi(z,w)|^2 d\lambda_\phi(z)\lambda_\phi(w).
\end{align*}
By Lemma \ref{lem-real}, we have
\begin{align}\label{I-1-I-2}
 \lim_{r\to\infty} I_1(r) = 0 \an  \lim_{r\to\infty} I_2(r) = 0.
\end{align}

\medskip

\noindent \textbf{ Claim A.} For any $r \ge r_{\mathfrak{p}, \mathfrak{q}}$, we have $I_1(n, r) \le I_1(r)$.

\medskip

Indeed,  by using the expression \eqref{finite-app} for $\Pi_n(z,w)$  and using the polar-coordinates system $z = \rho e^{i\alpha}, w = \sigma e^{i \beta}$ , we get
\begin{align*}
 | \Pi_n(\rho e^{i\alpha}, \sigma e^{i \beta} )|^2 =  \sum_{k, m=0}^{n-1} a_k^2  a_m^2 (\rho \sigma)^{k+m} e^{i(k -m)(\alpha-\beta)};
 \\
  | \Pi(\rho e^{i\alpha}, \sigma e^{i \beta} )|^2 =  \sum_{k, m=0}^{\infty} a_k^2  a_m^2 (\rho \sigma)^{k+m} e^{i(k -m)(\alpha-\beta)}.
\end{align*}
It follows that
\begin{align*}
 I_1(n,r) &=  4 \pi^2 \int_0^r   e^{-2 \phi(\rho)}  \rho d\rho \int_r^\infty   e^{-2 \phi(\sigma)} \sigma d \sigma  \cdot \frac{1}{\sigma^2}  \sum_{k=0}^{n-1} a_k^4  (\rho \sigma)^{2k} ;
 \\
 I_1(r) & =  4 \pi^2 \int_0^r    e^{-2 \phi(\rho)}   \rho d\rho \int_r^\infty   e^{-2 \phi(\sigma)} \sigma d \sigma  \cdot \frac{1}{\sigma^2}  \sum_{k=0}^{\infty} a_k^4  (\rho \sigma)^{2k}.
\end{align*}
Hence we have $I_1(n, r) \le I_1(r)$.

\medskip

\noindent \textbf{ Claim B.} For any $r \ge r_{\mathfrak{p}, \mathfrak{q}}$, we have $I_2(n,r ) \le I_2(r) + \frac{1}{r^2}$.

\medskip

Indeed, by using the polar-coordinates system and by using the identity
\[
\Big|  \frac{1}{\rho e^{i \alpha}} - \frac{1}{\sigma e^{i \beta}} \Big|^2 = \frac{1}{\rho^2} + \frac{1}{\sigma^2} -  \frac{1}{\rho \sigma} e^{i (\alpha- \beta) } -   \frac{1}{\rho \sigma} e^{- i (\alpha- \beta) },
\]
 we get
\begin{align*}
I_2(n,r) & =  4 \pi^2 \int_r^\infty   e^{-2 \phi(\rho)}  \rho d\rho \int_r^\infty   e^{-2 \phi(\sigma)} \sigma d \sigma  \cdot \underbrace{\Big[ \Big(\frac{1}{\rho^2} + \frac{1}{\sigma^2}\Big)  \sum_{k=0}^{n-1} a_k^4  (\rho \sigma)^{2k}  - \frac{2}{\rho \sigma} \sum_{k=0}^{n-2}  a_k^2 a_{k+1}^2 (\rho \sigma)^{2k+1} \Big]}_{\text{denoted by $S_n(\rho, \sigma)$}}.
\end{align*}
We can re-group the summands in $S_n(\rho, \sigma)$ in such a way that  in the new expression of  $S_n(\rho, \sigma)$, all summands are positive. Indeed, we have
\begin{align*}
S_n(\rho, \sigma)=&   \sum_{k=0}^{n-1}  a_k^4     \rho^{2k-2} \sigma^{2k}   +   \sum_{k=0}^{n-1}      a_{k}^4 \rho^{2k}\sigma^{2k-2} - \sum_{k=0}^{n-2}  2 a_k^2 a_{k+1}^2 (\rho \sigma)^{2k}
\\
=& a_0^4 \rho^{-2} + a_{n-1}^4 \rho^{2n-2} \sigma^{2n-4} +  \sum_{k=0}^{n-2}  \Big(    a_{k+1}^4     \rho^{2k} \sigma^{2k+2} +  a_{k}^4 \rho^{2k}\sigma^{2k-2} -   2 a_k^2 a_{k+1}^2 (\rho \sigma)^{2k}      \Big)
\\
 = & a_0^4 \rho^{-2} + a_{n-1}^4 \rho^{2n-2} \sigma^{2n-4} +  \sum_{k=0}^{n-2}  (   a_{k+1}^2     \rho^{k} \sigma^{k+1} -  a_{k}^2 \rho^{k}\sigma^{k-1})^2.
\end{align*}
It follows that
\begin{align}\label{I-2-n-new}
I_2(n,r) & =  4 \pi^2 \int_0^r   e^{-2 \phi(\rho)}  \rho d\rho \int_r^\infty   e^{-2 \phi(\sigma)} \sigma d \sigma  \cdot \Big[   a_0^4 \rho^{-2} + a_{n-1}^4 \rho^{2n-2} \sigma^{2n-4} +  \sum_{k=0}^{n-2}  (   a_{k+1}^2     \rho^{k} \sigma^{k+1} -  a_{k}^2 \rho^{k}\sigma^{k-1})^2\Big].
\end{align}
Similarly, we can express $I_2(r)$  in the following way:
\begin{align}\label{I-2-new}
I_2(r) & =  4 \pi^2 \int_0^r   e^{-2 \phi(\rho)}  \rho d\rho \int_r^\infty   e^{-2 \phi(\sigma)} \sigma d \sigma  \cdot \Big[   a_0^4 \rho^{-2} +  \sum_{k=0}^{\infty}  (   a_{k+1}^2     \rho^{k} \sigma^{k+1} -  a_{k}^2 \rho^{k}\sigma^{k-1})^2\Big].
\end{align}
Note that by definition, for any $n\ge 1$,
\begin{align*}
\frac{1}{a_{n-1}^2} = \| z^{n-1}\|_{L^2(\C,\, d\lambda_\phi)}^2 =  2 \pi \int_0^\infty   \rho^{2n-2}  e^{-2\phi(\rho)} \rho d\rho.
\end{align*}
Hence
\begin{align}\label{small-term}
\begin{split}
&  4 \pi^2 \int_r^\infty   e^{-2 \phi(\rho)}  \rho d\rho \int_r^\infty   e^{-2 \phi(\sigma)} \sigma d \sigma   \cdot   a_{n-1}^4 \rho^{2n-2} \sigma^{2n-4}
\\
&  \le  \Big( a_{n-1}^2 \cdot 2 \pi \int_0^\infty   \rho^{2n-2}  e^{-2\phi(\rho)} \rho d\rho\Big) \cdot \Big(  \frac{1}{r^2}    a_{n-1}^2 \cdot   2 \pi \int_0^\infty   \rho^{2n-2}  e^{-2\phi(\rho)} \rho d\rho\Big) = \frac{1}{r^2}.
\end{split}
\end{align}
Comparing \eqref{I-2-n-new} and \eqref{I-2-new}, taking \eqref{small-term} into account, we get the desired inequality
\[
I_2(n, r) \le I_2(r) + \frac{1}{r^2}.
\]

Finally, an application of \eqref{control-I-I} yields that
\begin{align}\label{hard-control}
 \sup_{n\in \N} \|  [ h_r^\infty,  \Pi_n] \|_{HS}^2 \le C (I_1(r) + I_2(r) + \frac{1}{r^2}).
\end{align}
The desired limit equality  \eqref{small-com}  now follows immediately from  \eqref{I-1-I-2} and \eqref{hard-control}.
\end{proof}

{\bf {Remark.}} Note that radial symmetry of the weight of our Fock space has been used in the proof of Claims A,B.

\subsection{Proof of item (iv) of Proposition \ref{prop-im}}\label{subsec-iv}

For any $R> r_{\mathfrak{p}, \mathfrak{q}}$,  denote
\[
\Psi^{(R)}_{\mathfrak{p}, \mathfrak{q}}(\X)  =  \exp\Big(   \int\limits_{r_{\mathfrak{p}, \mathfrak{q}} \le |z|\le R}    (  \kappa(\mathfrak{p}, z) - \kappa(\mathfrak{q}, z) )  \Pi(z,z)  d\lambda_\phi(z) \Big)   \prod_{x\in\X: | x| \le R}   \prod_{i= 1}^\ell\Big| \frac{x-p_i}{x- q_i }\Big|^2.
\]
Using Notation \eqref{many-notation}, we express the expectation $\E_{\PP_{\Pi}^{\mathfrak{q}}}[\Psi_{\mathfrak{p}, \mathfrak{q}}]$ as follows.

\begin{proposition}\label{prop-factor-general}
We have 
\[
\E_{\PP_{\Pi}^{\mathfrak{q}}}[\Psi_{\mathfrak{p}, \mathfrak{q}}] =  \Det_3(1 + T) \cdot \exp\Big( \sum_{i=1}^4 E_i(r) \Big),
\]
with $E_1(r), E_2(r),  E_3(r), E_4(r)$  given by
\begin{align*}
E_1(r) &=    \tr(\chi_0^r T ) - \frac{1}{2} \tr(\chi_0^r T^2)     - \frac{1}{2}  \tr(h_r^\infty  \Pi^{\mathfrak{q}}    h_0^r \Pi^{\mathfrak{q}}  ) ;
\\
E_2(r) &= \int_{  | z | \ge r}    \Big( h(z) - \frac{h(z)^2}{2}   + (\kappa(\mathfrak{p}, z) - \kappa(\mathfrak{q}, z)) \Big) \Pi(z,z)   d\lambda_\phi(z) ;
\\
E_3(r) &= \int_{ | z | \ge r}    \Big(  h(z) - \frac{h(z)^2}{2} \Big) (\Pi^{\mathfrak{q}}(z,z)    - \Pi(z,z)) d\lambda_\phi(z) ;
\\
E_4(r) & =  \frac{1}{4}\|  [ h_r^\infty,  \Pi^{\mathfrak{q}}] \|_{HS}^2.
\end{align*}
\end{proposition}

\begin{proof}
It suffices to prove that we have the following factorization
\begin{align}\label{factor-general}
\E_{\PP_{\Pi}^{\mathfrak{q}}}[\Psi^{(R)}_{\mathfrak{p}, \mathfrak{q}}] =  \Det_3(1 + T_{R}) \cdot \exp\Big( \sum_{i=1}^4 E_i(R; r) \Big),
\end{align}
with $E_1(R; r), E_2(R; r),  E_3(R; r), E_4(R; r)$  given by
\begin{align*}
E_1(R; r) &=    \tr(\chi_0^r T_{R} ) - \frac{1}{2} \tr(\chi_0^r T_{R}^2)     - \frac{1}{2}  \tr(h_r^R   \Pi^{\mathfrak{q}}    h_0^r \Pi^{\mathfrak{q}}  ) ;
\\
E_2(R; r) &= \int_{ r \le | z | \le R}    \Big( h(z) - \frac{h(z)^2}{2}   + (\kappa(\mathfrak{p}, z) - \kappa(\mathfrak{q}, z)) \Big) \Pi(z,z)   d\lambda_\phi(z) ;
\\
E_3(R; r) &= \int_{ r \le | z | \le R}    \Big(  h(z) - \frac{h(z)^2}{2} \Big) (\Pi^{\mathfrak{q}}(z,z)    - \Pi(z,z)) d\lambda_\phi(z) ;
\\
E_4(R; r) & =  \frac{1}{4}\|  [ h_r^R,  \Pi^{\mathfrak{q}}] \|_{HS}^2.
\end{align*}
The proof of factorization  \eqref{factor-general} is the same as that of  the factorization in Proposition \ref{prop-factor}.
\end{proof}

\begin{proof}[Proof of item (iv) of Proposition \ref{prop-im}]
The continuity  of the mapping $\mathfrak{p} \mapsto  \E_{\PP_{\Pi}^\mathfrak{q}}[\Psi_{\mathfrak{p}, \mathfrak{q}}]$ is immediate from the factorization in Proposition \ref{prop-factor-general} and the fact that $r = r_{\mathfrak{p}, \mathfrak{q}}$ depends continuously on $\mathfrak{p}, \mathfrak{q}$.
Fix any $\mathfrak{q}^0 = (q_1^0, \cdots, q_\ell^0)$ of distinct points of $\C$.  By the chain property of the Radon-Nikodym derivative,  we have
\[
\frac{d\PP_{\Pi}^{\mathfrak{p}}}{d\PP_{\Pi}^{\mathfrak{q}}} (\X) =    \frac{d\PP_{\Pi}^{\mathfrak{p}}}{d\PP_{\Pi}^{\mathfrak{q}^{0}}} (\X) \cdot \Big(  \frac{d\PP_{\Pi}^{\mathfrak{q}}}{d\PP_{\Pi}^{\mathfrak{q}^{0}}} (\X) \Big)^{-1}.
\]
In other words, we have
\[
\frac{\Psi_{\mathfrak{p}, \mathfrak{q}}(\X)}{\E_{\PP_{\Pi}^\mathfrak{q}}[\Psi_{\mathfrak{p}, \mathfrak{q}}]} = \frac{\Psi_{\mathfrak{p}, \mathfrak{q}^0}(\X)}{\E_{\PP_{\Pi}^{\mathfrak{q}^0}}[\Psi_{\mathfrak{p}, \mathfrak{q}^0}]}  \cdot  \frac{\E_{\PP_{\Pi}^{\mathfrak{q}^0}}[\Psi_{\mathfrak{q}, \mathfrak{q}^0}]}{\Psi_{\mathfrak{q}, \mathfrak{q}^0}(\X)}.
\]
Consequently,  the continuity obtained in item (iii) of Proposition  \ref{prop-im}, together with the  continuity of the mapping $\mathfrak{p} \mapsto  \E_{\PP_{\Pi}^\mathfrak{q}}[\Psi_{\mathfrak{p}, \mathfrak{q}}]$ implies the desired continuity of the mapping $(\mathfrak{p}, \mathfrak{q}) \mapsto  \E_{\PP_{\Pi}^\mathfrak{q}}[\Psi_{\mathfrak{p}, \mathfrak{q}}]$. Item (iv) of Proposition \ref{prop-im} is proved completely.
\end{proof}

{\bf Acknowledgements.}
We are deeply grateful to Alexei Klimenko for useful discussions. The research of A.~Bufetov on this project has received funding from the European Research Council (ERC) under the European Union's Horizon 2020 research and innovation programme under grant agreement No 647133 (ICHAOS). It has also been funded by the Grant MD 5991.2016.1 of the President of the Russian Federation, by  the Russian Academic Excellence Project `5-100'.  Y.~Qiu is supported by the grant IDEX UNITI - ANR-11-IDEX-0002-02, financed by Programme ``Investissements d'Avenir'' of the Government of the French Republic managed by the French National Research Agency.

This project was started as research in pairs at the Abbaye de L\'erins, Ile St. Honorat, and continued at the ICTP, Miramare, Trieste. We are deeply grateful to these institutions for their warm hospitality.

%\bibliography{mybib}
%\bibliographystyle{plain}

%\def\cprime{$'$} \def\cydot{\leavevmode\raise.4ex\hbox{.}}

\bigbreak
\noindent Alexander I. Bufetov\\
\noindent  Aix-Marseille Universit{\'e}, Centrale Marseille, CNRS, Institut de Math{\'e}matiques de Marseille, UMR7373, \\
  39 Rue F. Joliot Curie 13453, Marseille, France; \\
\noindent Steklov Institute of Mathematics, Moscow, Russia;\\
\noindent Institute for Information Transmission Problems, Moscow, Russia; \\
\noindent National Research University Higher School of Economics, Moscow, Russia.\\
\noindent  bufetov@mi.ras.ru, alexander.bufetov@univ-amu.fr

\bigbreak
\noindent Yanqi Qiu \\
\noindent
 CNRS, Institut de Math{\'e}matiques de Toulouse, Universit{\'e} Paul Sabatier, Toulouse, France.\\
\noindent yqi.qiu@gmail.com

\end{document}